\pdfoutput=1
\documentclass[11pt, letterpaper]{article}

\usepackage{amssymb}
\usepackage[all,cmtip]{xy}
\usepackage[width=17cm,top = 3.5cm,bottom=3.5cm]{geometry}
\usepackage{tikz}
\usepackage[all]{xy}
\usetikzlibrary{matrix,arrows,decorations.pathmorphing}
\usepackage{amsthm,amsfonts,amssymb,amsmath,amscd} 
\usepackage{aliascnt} 
\usepackage{mathrsfs} 
\usepackage{xargs,ifthen} 

\usepackage{comment}
\usepackage{color}

\usepackage{hyperref}
\definecolor{tocolor}{rgb}{.1,.1,.1}
\definecolor{urlcolor}{rgb}{.2,.2,.6}
\definecolor{linkcolor}{rgb}{.1,.1,.5}
\definecolor{citecolor}{rgb}{.4,.2,.1}
\hypersetup{backref=true, colorlinks=true, urlcolor=urlcolor, linkcolor=linkcolor, citecolor=citecolor}

\newcommandx{\thdef}[2]{
	\newaliascnt{#1}{theorem}  
	\newtheorem{#1}[#1]{#2}
	\aliascntresetthe{#1}  
	\newtheorem*{#1*}{#2}
	\expandafter\newcommand\expandafter{\csname #1autorefname\endcsname}{#2}
}

\makeatletter
\newtheorem*{rep@theorem}{\rep@title}
\newcommand{\newreptheorem}[2]{%
\newenvironment{rep#1}[1]{%
 \def\rep@title{#2 \ref{##1}}%
 \begin{rep@theorem}}%
 {\end{rep@theorem}}}
\makeatother

\newtheorem{theorem}{Theorem}[section]
\newreptheorem{theorem}{Theorem}

\thdef{lemma}{Lemma}
\thdef{corollary}{Corollary}
\thdef{conjecture}{Conjecture}
\newreptheorem{conjecture}{Conjecture}
\thdef{proposition}{Proposition}
\thdef{question}{Question}

\theoremstyle{definition}
\thdef{definition}{Definition}
\thdef{notation}{Notation}

\theoremstyle{remark}
\thdef{remark}{Remark}

\theoremstyle{remark}
\thdef{ex}{Example}

\newenvironment{example}
{\begin{ex}}%
{\hfill $\blacksquare$\end{ex}}

\newcommand{\spc}[1]{\mathsf{#1}} 
\newcommand{\shf}[1]{\mathcal{#1}} 

\newcommand{\rbrac}[1]{\left(#1\right)} 

\newcommandx{\fn}[2][2=]{#1\ifthenelse{\equal{#2}{}}{}{\!\rbrac{{#2}}}} 
\newcommandx{\id}[2][2=]{\fn{{\rm id}_{#1}}[#2]} 

\newcommand{\ext}[2][\bullet]{\spc{\Lambda}^{#1}{#2}} 
\newcommandx{\End}[2][1=]{\fn{\spc{End}_{#1}}[#2]} 
\newcommandx{\Hom}[2][1=]{\fn{\spc{Hom}_{#1}}[#2]} 
\newcommandx{\Aut}[2][1=]{\fn{\spc{Aut}_{#1}}[#2]} 
\newcommandx{\image}[1]{\fn{\spc{img}}[#1]} 
\renewcommandx{\ker}[1]{\fn{\spc{ker}}[#1]} 
\newcommandx{\rank}[1]{\fn{\mathrm{rank}}[#1]} 

\newcommandx{\ann}[1]{\fn{\spc{ann}}[#1]} 

\newcommandx{\hlgy}[3][1=\bullet,3=]{\spc{H}_{#1}^{#3}\!\rbrac{{#2}}} 
\newcommandx{\cohlgy}[3][1=\bullet,3=]{\spc{H}^{#1}_{#3}\!\rbrac{{#2}}} 
\newcommandx{\chow}[3][1=\bullet,3=]{\spc{A}^{#1}_{#3}\!\rbrac{{#2}}} 

\newcommandx{\Ext}[3][1=\bullet,3=]{\fn{\spc{Ext}^{#1}_{#3}}[{#2}]} 
\newcommandx{\Tor}[3][1=\bullet,3=]{\fn{\spc{Tor}^{#1}_{#3}}[{#2}]} 

\newcommandx{\Pic}[1]{\fn{\spc{Pic}}[{#1}]} 
\newcommandx{\chernalg}[2][1=\bullet]{\fn{\spc{Chern}^{#1}}[{#2}]} 
\newcommandx{\chern}[2][1=]{\fn{c_{#1}}[#2]} 
\newcommandx{\ch}[2][1=]{\fn{\mathrm{ch}_{#1}}[{#2}]} 

\newcommandx{\sKer}[2][1=]{ \fn{ \shf{K}er_{#1}}[{#2}] } 
\newcommandx{\sHom}[2][1=]{ \fn{ \shf{H}om_{#1}}[{#2}] } 
\newcommandx{\sEnd}[2][1=]{ \fn{ \shf{E}nd_{#1}}[{#2}] } 
\newcommandx{\sExt}[3][1=\bullet,3=]{\fn{\shf{E}xt^{#1}_{#3}}[{#2}]} 
\newcommandx{\sTor}[3][1=\bullet,3=]{\fn{\shf{T}or^{#1}_{#3}}[{#2}]} 

\newcommandx{\forms}[2][1=\bullet]{\Omega^{#1}_{#2}} 
\newcommandx{\can}[1][1=]{\omega_{#1}} 
\newcommandx{\acan}[1][1=]{\omega_{#1}^{-1}} 
\newcommandx{\tshf}[1]{\shf{T}_{#1}} 
\newcommandx{\mvect}[2][1=\bullet]{ \ext[#1]{\tshf{#2}} }
\newcommandx{\der}[2][1=\bullet]{\mathscr{X}^{#1}_{#2}} 
\newcommandx{\sJet}[3][1=,2=]{\shf{J}^{#1}_{#2}#3} 

\newcommandx{\tb}[2][1=]{\spc{T}_{\!#1}{#2}} 
\newcommandx{\ctb}[2][1=]{\spc{T}_{\!#1}^*{#2}} 
\newcommandx{\lie}[2][2=]{\fn{\mathscr{L}_{#1}}[#2]} 
\newcommandx{\hook}[2][2=]{\fn{i_{#1}}[#2]} 

\newcommand{\Rep}[1]{\fn{\cat{R}ep}[#1]}

\makeatletter
\newcommand{\thickbar}{\mathpalette\@thickbar}
\newcommand{\@thickbar}[2]{{#1\mkern1.5mu\vbox{
  \sbox\z@{$#1\mkern-1mu#2\mkern-1mu$}%
  \sbox\tw@{$#1\overline{#2}$}%
  \dimen@=\dimexpr\ht\tw@-\ht\z@-.6\p@\relax
  \hrule\@height.4\p@ 
  \vskip1\p@
  \hrule\@height.4\p@ 
  \vskip\dimen@
  \box\z@}\mkern1.5mu}
}
\makeatother

\def\Rep{\text{Rep}}

\def\ker{\text{ker}}

\def\End{\text{End}}
\def\log{\text{log}}

\numberwithin{equation}{section}
\newtheoremstyle{parag}
  {\topsep}   
  {\topsep}   
  {}  
  {}       
  {\bfseries} 
  {.}         
  { } 
  {}          
\theoremstyle{parag}

\usepackage{MnSymbol}

\makeatletter
\def\@cite#1#2{{\normalfont[{#1\if@tempswa , #2\fi}]}}
\makeatother

\renewcommand{\Pic}{\mathrm{Pic}}

\def\Aut{\text{Aut}}

\begin{document}

\title{\vspace{-4em} \huge Normal forms and moduli stacks for logarithmic flat connections}

\date{}

\author{
Francis Bischoff\thanks{Exeter College and Mathematical Institute, University of Oxford; {\tt francis.bischoff@maths.ox.ac.uk }}
}
\maketitle
\abstract{We establish normal form theorems for a large class of singular flat connections on complex manifolds, including connections with logarithmic poles along weighted homogeneous Saito free divisors. As a result, we show that the moduli spaces of such connections admit the structure of algebraic quotient stacks. In order to prove these results, we introduce homogeneous Lie groupoids and study their representation theory. In this direction we prove two main results: a Jordan-Chevalley decomposition theorem, and a linearization theorem. We give explicit normal forms for several examples of free divisors, such as homogeneous plane curves, reductive free divisors, and one of Sekiguchi's free divisors.}

\tableofcontents

\section{Introduction}
In this paper we obtain normal form theorems for a large class of singular flat connections on complex manifolds, including connections with logarithmic poles along weighted homogeneous Saito free divisors. The simplest instance of these connections corresponds to linear ordinary differential equations with a Fuchsian singularity at the origin: 
\[
z \frac{ds}{dz} = A(z) s(z). 
\]
These equations have been the subject of a large literature (e.g. \cite{hukuhara1937proprietes, turrittin1955convergent, gantmacher1959theory, alllevelt1961hypergeometric, levelt1975jordan, babbitt1983formal, kleptsyn2004analytic, boalch2011riemann, bischoff2020lie}). 
In this case we recover the standard normal form theorem, which states that there is a holomorphic gauge transformation converting $A$ to a matrix of the following form:
\[
S + \sum_{i \geq 0} N_{i}z^{i},
\]
where $S$ is a constant semisimple matrix, and $N_{i}$ are constant nilpotent matrices which satisfy $[S,N_{i}] = i N_{i}$. In the above expression, $S + N_{0}$ is the \emph{linear approximation} to the holomorphic matrix $A(z)$, and the $N_{i}$ are the higher order \emph{resonant correction terms}. 

The geometry of Fuchsian singularities was studied in \cite{boalch2011riemann}, where the moduli space of framed logarithmic $G$-connections on the disc with residue in a fixed adjoint orbit was shown to admit a quasi-Hamiltonian structure and to be related to the multiplicative Brieskorn-Grothendieck-Springer resolution. As explained in \cite{SAFRONOV2016733}, this can be interpreted as saying that the moduli space of (unframed) logarithmic $G$-connections is Lagrangian in the $1$-shifted symplectic stack $[G/G]$. In this paper, we show that our moduli spaces of singular flat connections admit the structure of algebraic quotient stacks. To our knowledge, these moduli stacks have not appeared elsewhere in the literature. We expect them to admit interesting geometric structures, and to be amenable to study via the methods of non-reductive GIT \cite{MR2330155, MR3989432}, but we postpone such a detailed investigation of their geometry to future work.

As an example, consider the singular hypersurface $D \subset \mathbb{C}^3$ defined as the vanishing locus of $F_{B,5} = xy^4 + y^3 z + z^3$. A real slice of this surface is pictured below. 

\[ \centering
\includegraphics[scale=0.04]{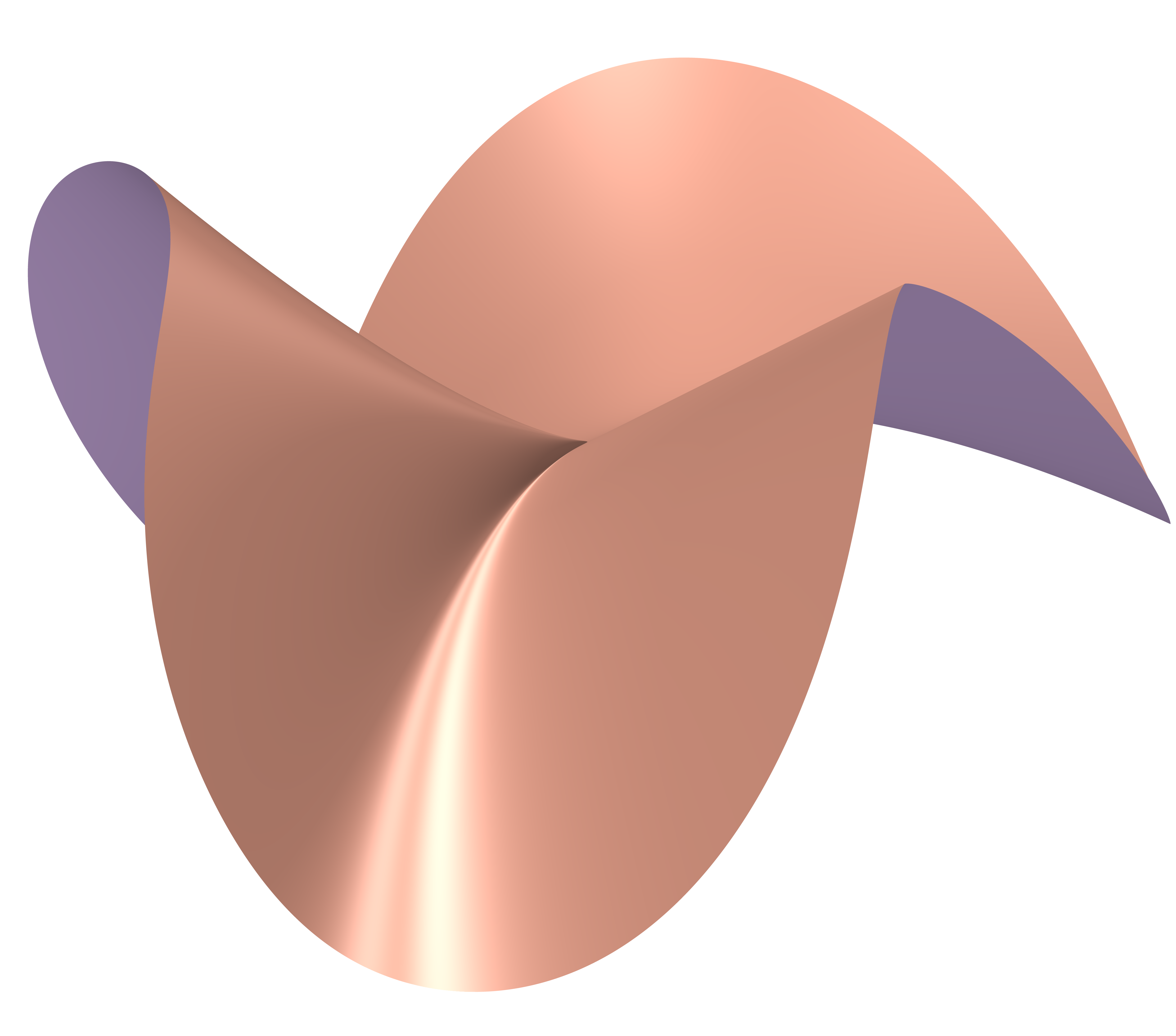}
\]

This is one of the 17 free divisors in Sekiguchi's classification \cite{MR2588504}. It is weighted homogeneous of degree $9$ with respect to the vector field 
\[
E = x \partial_{x} + 2y \partial_{y} + 3 z \partial_{z}. 
\]
The logarithmic tangent bundle $T_{\mathbb{C}^3}(-\log D)$, defined as the sheaf of vector fields tangent to the smooth locus of $D$, is locally free. Let $E, V, W$ denote a basis of logarithmic vector fields, and let $\alpha, \beta, \gamma \in \Omega^{1}_{\mathbb{C}^3}(\log D)$ denote the dual basis of logarithmic $1$-forms (see Sections \ref{SekiDiv} and \ref{Sekinormalform} for more details). A connection on $\mathbb{C}^3$ with logarithmic singularities along $D$ has the following form
\[
\nabla = d + A \otimes \alpha + B \otimes \beta + C \otimes \gamma, 
\]
where $A, B, C : \mathbb{C}^3 \to \mathfrak{h}$ are holomorphic maps valued in a reductive Lie algebra. In Section \ref{Sekinormalform} we prove the following normal form theorem.

\begin{theorem} \label{introtheorem}
Let $H$ be a connected complex reductive group with Lie algebra $\mathfrak{h}$, and let $\nabla = d + \omega$ be a flat connection on a principal $H$-bundle over $\mathbb{C}^3$ with logarithmic singularities along $D$. There exists a holomorphic gauge transformation which brings the connection into the following form 
\[
\nabla = d + (S + N) \otimes \alpha + (B - \frac{32}{3}xN) \otimes \beta + (C - 4y N) \otimes \gamma, 
\]
where $S \in \mathfrak{h}$ is a semisimple element, $N : \mathbb{C}^3 \to \mathfrak{h}$ is a holomorphic family of nilpotent elements, and $B, C : \mathbb{C}^3 \to \mathfrak{h}$ are holomorphic maps. Furthermore, these data satisfy the following equations 
\begin{enumerate}
\item $E(B) = B + [B, S]$ \label{inseki1}
\item $E(C) = 2C + [C,S] $ \label{inseki2} 
\item $E(N) = [N, S]$ \label{inseki3}
\item $V(N) = [N,B]$ \label{inseki4} 
\item $W(N) = [N,C]$ \label{inseki5} 
\item $V(C) - W(B) = 24z S + 6yB - 40 x C - [B,C].$ \label{inseki6}
\end{enumerate}
\end{theorem}

A closer look at Equations \ref{inseki1}, \ref{inseki2}, and \ref{inseki3} in the above Theorem reveals that they are eigenvalue equations for the weight components of $N, B$ and $C$. Consequently, these must be polynomial functions of a fixed bounded degree, and they must lie in a finite dimensional affine space $W_{S}$. The remaining Equations \ref{inseki4}, \ref{inseki5} and \ref{inseki6} (along with the nilpotency of $N$) then define an algebraic subvariety $X_{S} \subset W_{S}$ which parametrises flat connections in normal form. If we fix the \emph{semisimple residue} $S$, which is uniquely determined up to the adjoint action, then the normal form in the Theorem is itself unique up to the action of a group $Aut(S)$ of residual gauge symmetries preserving $S$. We study this group in Section \ref{studyingsymmetries}, and show that it is an algebraic group with an algebraic action on $X_{S}$. Therefore, we conclude that the `moduli space' of flat connections with semisimple residue in the adjoint orbit of $S$ is the algebraic quotient stack 
\[
[X_{S}/Aut(S)].
\]
This result is established for a large class of singular flat connections in Theorems \ref{MainClassificationTheorem} and \ref{AlgebraicModuliStack}.

In order to prove the normal form theorems we introduce \emph{homogeneous groupoids} and study their representation theory. These are holomorphic Lie groupoids that are closely related to torus actions and hence can be studied using similar tools, but they also exhibit a range of interesting new behaviours. There is an abundance of examples of these groupoids. For instance, they arise from rational representations of complex algebraic groups which have a unique s-equivalence class. However, the most important examples are the twisted fundamental groupoids of weighted homogeneous Saito free divisors $(X,D)$, since the representations of these correspond to flat connections on $X$ with logarithmic singularities along $D$. There is a vast literature providing us with a rich source of such examples, such as plane curve singularities \cite{saito1980theory}, linear and reductive free divisors \cite{MR2228227, MR2521436, MR2795728, MR3237442, Linearfreeunpub}, and Sekiguchi's family of $17$ free divisors in dimension $3$ \cite{MR2588504}. 

In Section \ref{ReptheoryWHG} we establish the main structure theorems for representations of homogeneous groupoids. In Theorem \ref{JCdecomposition} we prove a Jordan-Chevalley decomposition theorem for representations, which allows us to factor out the unipotent component of the monodromy. This is a generalization of the Jordan decomposition for Fuchsian singularities due to Hukuhara \cite{MR5229}, Turrittin \cite{turrittin1955convergent} and Levelt \cite{levelt1975jordan}, and reformulated from the perspective of Lie groupoids in \cite{bischoff2020lie}. In Section \ref{LinearizationSection} we specialize to action groupoids arising from linear representations of reductive groups, and in Theorem \ref{mainlinearizationthm} we show that the representations of these groupoids admit linearizations if their monodromy is semisimple. This later theorem may be applied directly in the setting of flat connections with logarithmic singularities along reductive free divisors, and as a result, we obtain in Corollary \ref{redfreedivlinearization} a linearization theorem for these connections. 

In Section \ref{normalformtheorems} we apply the structure theory of homogeneous groupoids to obtain the normal form theorems announced at the beginning of the introduction. More precisely, in Theorem \ref{MainClassificationTheorem} we give a functorial classification for representations of homogeneous groupoids and their Lie algebroids. This implies in particular that every representation may be put into a normal form which, just as in the case of Fuchsian singularities, consists of a linear approximation to the representation along with higher order resonant correction terms. This also leads to an explicit description of the moduli space of representations, and we show in Theorem \ref{AlgebraicModuliStack} that it is a quotient stack in the category of affine algebraic varieties. In Section \ref{normalformexplicitsection}, we write down the explicit normal forms for flat connections with logarithmic singularities along Saito free divisors in the case of homogeneous plane curves, linear free divisors, and one of Sekiguchi's free divisors (Theorem \ref{introtheorem} above). These are respectively given in Theorems \ref{planecurvesingnormalform}, \ref{linearalgrepnormalform} and \ref{sekinormalform}. Finally, as a further application, we show in Corollary \ref{notaction} that the logarithmic tangent bundle of the vanishing locus $F_{B,5}$ considered above is not isomorphic to an action algebroid. 

\vspace{.05in}

\noindent \textbf{Conventions.} 
In defining the Lie algebra of a Lie group, it is customary to take left invariant vector fields. On the other hand, the Lie algebroid of a Lie groupoid is usually defined using right invariant vector fields. In order to maintain consistency, in this paper we have decided to go with the latter convention, even in the case of Lie groups. This has the slightly annoying consequence that the Lie bracket of a matrix Lie algebra is the negative of the commutator. For this reason, we will denote the commutator bracket with a subscript $c$, so that $[X, Y]_{c} = -[X,Y]. $

\vspace{.05in}

\noindent \textbf{Acknowledgements.} 
I would like to thank M. Gualtieri for many fruitful conversations and for suggesting several improvements to the paper, P. Levy for providing the proof of Proposition \ref{proofofreductive}, and L. Narv\'{a}ez Macarro for several useful correspondences and for sharing with me the unpublished paper \cite{Narvaez1}.

\section{Representation theory of holomorphic Lie groupoids} \label{Firstsection}
In this paper we will be studying the representation theory of a holomorphic Lie groupoid $\mathcal{G}$ over a complex manifold $X$. In this section we review some of the basic features of this representation theory. For a more detailed treatment, see \cite{bischoff2020lie, gualtieri2018stokes}. Throughout this paper, the groupoid representations will be valued in a structure group $H$, which is taken to be a connected complex reductive Lie group. We denote its Lie algebra $\mathfrak{h}$. 

Let $P$ be a right principal $H$-bundle over $X$. The \emph{Atiyah groupoid} of $P$ is the Lie groupoid over $X$ consisting of $H$-equivariant maps between the fibres of $P$. It is given by 
\[
\mathcal{G}(P) = Pair(P)/H.
\]
Its isotropy groupoid is the \emph{gauge groupoid} $Aut_{H}(P)$, consisting of $H$-equivariant automorphisms of the fibres. This groupoid may be constructed as an associated bundle as follows: 
\[
Aut_{H}(P) = (P \rtimes H)/H,
\]
where the action of $H$ is given by 
\[
(p, h) \ast k = (pk, k^{-1} h k). 
\]
The sections of $Aut_{H}(P)$ are known as \emph{gauge transformations}. If $P$ is the trivial bundle, then the Atiyah groupoid is simply $Pair(X) \times H$, and the gauge groupoid is $X \times H$. Changes of trivialization then act by conjugating the $H$-factor. 

The Lie algebroid of $\mathcal{G}(P)$ is the \emph{Atiyah algebroid}, whose sections consist of $H$-invariant vector fields on $P$. It is given by 
\[
At(P) = TP/H. 
\]
When $P$ is trivial, it is isomorphic to $TX \oplus \mathfrak{h}$. The Lie algebroid of the gauge groupoid is the associated bundle of Lie algebras 
\[
\mathfrak{aut}_{H}(P) = (P \rtimes \mathfrak{h})/H.
\]
When $P$ is trivial, it is isomorphic to $X \times \mathfrak{h}$.

An \emph{$H$-representation} of a groupoid $\mathcal{G}$ is defined by the data of a principal $H$-bundle $P$ and a homomorphism 
\[
\phi: \mathcal{G} \to \mathcal{G}(P). 
\]
When $P$ is the trivial bundle the representation is defined by a homomorphism $\mathcal{G} \to H$. The $H$-representations form a category which we denote $Rep(\mathcal{G}, H)$. 

Let $A$ be the Lie algebroid of $\mathcal{G}$. Any $H$-representation $\phi$ of $\mathcal{G}$ may be differentiated to give a representation of $A$:
\[
d\phi : A \to At(P). 
\]
In this way, we obtain a functor 
\[
Rep(\mathcal{G}, H) \to Rep(A, H),
\]
where $Rep(A, H)$ is the category of $H$-representations of the Lie algebroid $A$. When $\mathcal{G}$ is source-simply connected, Lie's second theorem \cite{mackenzie2000integration, moerdijk2002integrability} implies that this functor gives an isomorphism of categories. 

Let $(P, \phi)$ be a representation of $\mathcal{G}$. The Atiyah groupoid of $P$ acts on the gauge groupoid $Aut_{H}(P)$ by conjugation, and the map $\phi$ induces an action of $\mathcal{G}$ on $Aut_{H}(P)$. By differentiating this action, we also obtain an action of $\mathcal{G}$ on the bundle of Lie algebras $\mathfrak{aut}_{H}(P)$. An automorphism of $\phi$ is equivalent to a section of $Aut_{H}(P)$ (i.e. a gauge transformation), which is invariant under the $\mathcal{G}$-action. Similarly, an endomorphism (or infinitesimal automorphism) of $\phi$ is an invariant section of $\mathfrak{aut}_{H}(P)$. 

\section{Homogeneous groupoids} \label{whg}
In this section, we introduce \emph{homogeneous groupoids}, which are the main object of study in this paper. We begin with the definitions and then give examples. 

Let $X$ be a complex manifold equipped with a holomorphic action of $(\mathbb{C}^*)^{k}$. We may form the action groupoid $(\mathbb{C}^*)^{k} \ltimes X$, a groupoid over $X$ with the following structure maps:
\begin{align*}
t(\lambda, x) = \lambda \ast x, \qquad s(\lambda, x) = x, \\
m((\lambda, \mu \ast x), (\mu, x) )= (\lambda \mu, x). 
\end{align*}
The source simply connected cover is the action groupoid $\mathbb{C}^{k} \ltimes X$, where $\mathbb{C}^{k}$ acts via the exponential map. The kernel of the exponential map is the lattice $2 \pi i \mathbb{Z}^k \subset \mathbb{C}^k$. Therefore we have the injective map
\[
j: \mathbb{Z}^{k} \times X \to \mathbb{C}^{k} \ltimes X,
\]
whose image is purely contained in the isotropy. We also have a natural projection homomorphism 
\[
p : \mathbb{C}^{k} \ltimes X \to \mathbb{C}^k, \qquad (\lambda, x) \mapsto \lambda. 
\]
Consider a source-connected holomorphic Lie groupoid $\mathcal{G}$ over $X$ which is equipped with two morphisms 
\begin{itemize}
\item $i : \mathbb{C}^{k} \ltimes X \to \mathcal{G}$, 
\item $\pi : \mathcal{G} \to \mathbb{C}^k$,
\end{itemize}
satisfying $\pi \circ i = p$. We have the inclusion $i \circ j : \mathbb{Z}^k \times X \to \mathcal{G}$ as an isotropic subgroupoid.

\begin{definition}
The data $(\mathcal{G}, i, \pi)$ is called \emph{central} if $\mathbb{Z}^k \times X $ is a central subgroupoid of $\mathcal{G}$, meaning that for all $n \in \mathbb{Z}^k$ and $g \in \mathcal{G}$ the following identity holds:
\[
(n, t(g)) \ast g = g \ast (n, s(g)). 
\]
\end{definition}

The following result gives a practical way of checking the centrality condition. 
\begin{lemma}  \label{centralitycrit}
Let $A$ be the Lie algebroid of $\mathcal{G}$ and suppose that the anchor map $\rho: A \to TX$ is generically injective. Then $(\mathcal{G}, i, \pi)$ is central. 
\end{lemma}
\begin{proof}
By assumption, $\rho$ is injective on a dense subset $U \subseteq X$, which is automatically open and saturated. Let $O \subseteq U$ be an orbit, and let $s^{-1}(x) \subseteq \mathcal{G}$ be a source fibre for $x \in O$. Then $t: s^{-1}(x) \to O$ is a principal $\mathcal{G}(x,x)$-bundle, where $\mathcal{G}(x,x)$ is the isotropy group at $x$. This group is discrete since $\rho = dt|_{s^{-1}(x)}$ is injective. The restriction of $\mathcal{G}$ to the orbit $O$ is isomorphic to the Atiyah groupoid of $s^{-1}(x)$, namely $\mathcal{G}|_{O} \cong \mathcal{G}(s^{-1}(x))$. Hence the isotropy subgroupoid over $O$ is isomorphic to the gauge groupoid 
\[
Aut_{\mathcal{G}(x,x)}(s^{-1}(x)).
\]
This is a locally trivial bundle of discrete groups. In particular, it is a covering space of $O$ and hence satisfies the path lifting property. 

Now given $g \in s^{-1}(x)$, we can consider a path $g_{t}: [0,1] \to s^{-1}(x)$, such that $g_{0} = id_{x}$ and $g_{1} = g$. Given $n \in \mathbb{Z}^k$, consider the path of isotropy elements $g_{t}(n, x) g_{t}^{-1}$ which goes from $(n, x)$ to $g(n,x)g^{-1}$, and which satisfies $t(g_{t}(n, x) g_{t}^{-1}) = t(g_{t})$. Consider as well the path $(n, t(g_{t}))$, which also starts at $(n,x)$ and covers $t(g_{t})$. By unique path lifting, these two paths must coincide, showing that 
\[
g(n,x)g^{-1} = (n, t(g)). 
\]
Hence, $\mathbb{Z}^k$ is central in $\mathcal{G}|_{U}$. By continuity, $\mathbb{Z}^k$ must be central in $\mathcal{G}$. This works even when $\mathcal{G}$ is not Hausdorff since we can restrict to a Hausdorff neighbourhood of the identity bisection. 
\end{proof}

The decomposition of $X$ into the orbits of $\mathcal{G}$ defines an equivalence relation on $X$. We define a `weaker' equivalence relation, where two points $x, y \in X$ are said to be \emph{s-equivalent}, denoted $x \sim y$, if the closures of their orbits intersect (i.e. we take the equivalence relation generated by this relation). 

\begin{definition}
The data $(\mathcal{G}, i, \pi)$ defines a \emph{homogeneous groupoid} if it is central and if $X$ has a unique s-equivalence class. 
\end{definition}

There are two important special cases where $X$ has a unique s-equivalence class: 
\begin{enumerate}
\item There is a point $x_{0} \in X$ which is in the closure of all orbits. 
\item There is a dense orbit. 
\end{enumerate}

Suppose that $X = V$ is a vector space and that the action of $(\mathbb{C}^*)^{k}$ is linear, so that $V$ is a representation. A homogeneous groupoid $(\mathcal{G}, i, \pi)$ is \emph{positive} if there is a factor $\mathbb{C}^{*} \subseteq (\mathbb{C}^*)^{k}$ whose action on $V$ has only positive weights. Note that in this case, $V$ automatically has a unique s-equivalence class because the origin lies in the closure of every orbit. 

We end by noting the following immediate consequence of the definition of a homogeneous groupoid. 
\begin{proposition} \label{liftedactionongpd}
Let $(\mathcal{G}, i, \pi)$ be a homogeneous groupoid. Then the conjugation action of $\mathbb{C}^{k}$ descends to an action of $(\mathbb{C}^{*})^{k}$ on $\mathcal{G}$ by Lie groupoid automorphisms. Furthermore, the morphism $\pi$ is invariant under this action. 
\end{proposition}

\subsection{Examples: Representations of algebraic groups}  \label{algebraicgroupex}
Let $G$ be a connected complex linear algebraic group with Lie algebra $\mathfrak{g}$, let $V$ be a representation, and let $G \ltimes V$ be the associated action groupoid, which has Lie algebroid $\mathfrak{g} \ltimes V$. The source simply connected integration of $\mathfrak{g} \ltimes V$ is given by $\mathcal{G} = \tilde{G} \ltimes V$, where $\tilde{G}$ is the universal cover of $G$.

By a result of Mostow (e.g. see page 158 of \cite{borel2012linear}) $G$ has a Levi decomposition $G \cong U \rtimes L,$ where $U$ is a connected unipotent normal subgroup (the unipotent radical of $G$), and $L$ is a connected reductive group. Up to a finite covering, the reductive group $L$ has the form $(\mathbb{C}^*)^k \times S,$ where $S$ is a semisimple group (which we can assume is simply connected). Since $U$ is contractible, it follows that the universal cover of $G$ is given by $\tilde{G} \cong U \rtimes (\mathbb{C}^k \times S).$ From this expression, it is clear that $\tilde{G}$ contains a subgroup $\mathbb{C}^k$ which acts on $V$ through its quotient $(\mathbb{C}^*)^k$ via the exponential map. Furthermore, the subgroup $2 \pi i \mathbb{Z}^k \subset \mathbb{C}^k$ is central in $\tilde{G}$ and acts trivially on $V$. We therefore get an inclusion of groupoids 
\[
i : \mathbb{C}^k \ltimes V \to \tilde{G} \ltimes V,
\]
as well as a map $\pi : \tilde{G} \ltimes V \to \mathbb{C}^k$ given by the composition of projections
\[
\tilde{G} \ltimes V \to \tilde{G} =  U \rtimes (\mathbb{C}^k \times S) \to  (\mathbb{C}^k \times S) \to \mathbb{C}^k,
\]
and it is clear that $\pi \circ i = p$. The data $(\mathcal{G}, i, \pi)$ is central and hence defines a homogeneous groupoid if $V$ has a unique s-equivalence class. This is guaranteed if we assume that $G$ has an open orbit, or that $(\mathbb{C}^*)^k$ contains a factor $\mathbb{C}^*$ whose action on $V$ has only positive weights. In the later case, $(\mathcal{G}, i, \pi)$ is a positive homogeneous groupoid. 

\subsection{Examples: Saito free divisors} \label{logconnectionschapt}
We recall the notion of free divisors, which is due to K. Saito \cite{saito1980theory}. Let $X$ be a complex manifold, let $D \subset X$ be a reduced divisor, and let $I_{D}$ be its reduced defining ideal. A holomorphic vector field $V$ is \emph{logarithmic} if it preserves $I_{D}$. Equivalently, $V$ is logarithmic if it is tangent to the smooth locus of $D$. We denote by $T_{X}(-\log D)$ the sheaf of logarithmic vector fields. It is a coherent subsheaf of the tangent bundle which is closed under the Lie bracket. In general $T_{X}(-\log D)$ is not locally free, and when it is we call $D$ a \emph{free divisor}. Saito gave the following simple criterion for determining whether a divisor is free. 
\begin{theorem}[Saito's criterion \cite{saito1980theory}] \label{SaitoCrit}
A hypersurface $D \subset \mathbb{C}^n$ is a free divisor in a neighbourhood of a point $p$ if and only if there are $n$ germs of logarithmic vector fields 
\[
V_{i} = \sum_{j} a_{ij}(z) \partial_{z_{j}}, \ \ i = 1, ..., n,
\]
where $(z_{1}, ..., z_{n})$ are the standard coordinates on $\mathbb{C}^n$, such that $det(a_{ij}(z))$ is a reduced equation for $D$ around $p$. In this case, the vector fields $V_{i}$ form a basis of $T_{X}(-\log D)$. 
\end{theorem}
We now assume that $D$ is a free divisor. In this case, $T_{X}(-\log D)$ is a Lie algebroid, called the \emph{logarithmic tangent bundle} of the pair $(X,D)$. By definition, its anchor map is an injection on sections. Therefore, by a criterion of Debord \cite[Theorem 2]{debord2001holonomy} (see also \cite{crainic2003integrability}), it is integrable to a Lie groupoid. There is a unique integration $\Pi(X,D)$ with simply connected source fibres \cite{moerdijk2002integrability}. It is called the \emph{twisted fundamental groupoid} of the pair $(X,D)$. 

Representations of $T_{X}(- \log D)$ are called flat connections on $X$ with logarithmic singularities along $D$ (also called flat logarithmic connections when $D$ is understood). As explained in Section \ref{Firstsection} these are equivalent to representations of $\Pi(X,D)$.

Suppose that $X$ is connected and equipped with an action of $\mathbb{C}^*$, and that $I_{D}$ is generated by a global holomorphic function $f: X \to \mathbb{C}$. Let $E$ be the \emph{Euler vector field}, which is the generator of the $\mathbb{C}^*$-action. We say that $D$ is a \emph{weighted homogeneous free divisor} if $E(f) = c f$, for some non-zero constant $c$. In this case, $E \in T_{X}(-\log D)$ and the action of $\mathbb{C}^*$ preserves $D$. 

\begin{theorem} \label{whgconst}
Let $D \subset X$ be a weighted homogeneous free divisor. Then its twisted fundamental groupoid $\Pi(X,D)$ is a homogeneous Lie groupoid.
\end{theorem}
\begin{proof}
This proof proceeds by checking the conditions. The $\mathbb{C}^*$-action on $X$ defines an action groupoid $\mathbb{C}^{*} \ltimes X$. Its source simply connected cover is given by $\mathbb{C} \ltimes X$, and its Lie algebroid is given by $\mathbb{C} \times X$, with anchor map $\rho(\lambda ,x) = \lambda E_{x}.$ Since $E \in T_{X}(-\log D)$, the anchor map factors through a Lie algebroid morphism $\mathbb{C} \times X \to T_{X}(-\log D)$, and by Lie's second theorem \cite{mackenzie2000integration, moerdijk2002integrability} this integrates to a Lie groupoid morphism $i : \mathbb{C} \ltimes X \to \Pi(X,D).$ 

The meromorphic $1$-form $d\log f$ defines a global holomorphic algebroid $1$-form, namely a section of $\Omega^{1}_{X}(\log D).$ Because it is closed, it defines a Lie algebroid morphism 
\[
\frac{1}{c} d \log f: T_{X}(-\log D) \to \mathbb{C},
\]
which integrates by Lie's second theorem to a groupoid morphism $\pi: \Pi(X,D) \to \mathbb{C}.$
The derivative of the composite $\pi \circ i$ is the Lie algebroid morphism: 
\[
\mathbb{C} \times X \to \mathbb{C}, \qquad (\lambda, x) \mapsto \frac{\lambda}{c} d \log f(E_{s}) = \lambda. 
\]
Hence $\pi \circ i : \mathbb{C} \ltimes X \to \mathbb{C}$ is the canonical projection map.

Let $U = X \setminus D$, which is a connected open dense subset of $X$. The anchor map of $T_{X}(-\log D)$ is an isomorphism over $U$, and hence $(\Pi(X,D), i, \pi)$ is central by Lemma \ref{centralitycrit}. Furthermore, $U$ is an orbit of $\Pi(X,D)$ which is dense, implying that there is a unique s-equivalence class. Therefore, $(\Pi(X,D), i, \pi)$ is a homogeneous groupoid. 
\end{proof}
\begin{remark}
The most important examples of weighted homogeneous free divisors arise when $X$ is a vector space equipped with a linear action of $\mathbb{C}^*$ which has only positive weights. In this case, the homogeneous groupoid $(\Pi(X,D), i, \pi)$ constructed in Theorem \ref{whgconst} is positive. 
\end{remark}

\subsubsection{Plane curve singularities} \label{planecurvesingularities}
Let $D \subset \mathbb{C}^2$ be a plane curve defined by a reduced holomorphic function $f$.  By \cite[Corollary 1.7]{saito1980theory} $D$ is a free divisor. Choose relatively prime positive integers $p$ and $q$, and consider the weighted Euler vector field 
\[
E = p x \partial_{x} + q y \partial_{y}. 
\]
We assume that $f$ is weighted homogeneous of degree $n$, meaning that $E(f) = nf$. This implies that $f$ is a linear combination of monomials of the form $x^{k}y^{r}$, where $kp + rq = n$. Since $E \in T_{\mathbb{C}^2}(-\log D)$, it follows that the only singularity of $D$ is at the origin. Define the following vector field
\[
V = -(\partial_{y}f) \partial_{x} + (\partial_{x} f) \partial_{y}. 
\]
Then, since $V(f) = 0$, it follows that $V \in T_{\mathbb{C}^2}(-\log D)$. Taking the wedge product with $E$, we get $E \wedge V = E(f) \partial_{x} \wedge \partial_{y} = nf \partial_{x} \wedge \partial_{y}.$ Hence by Theorem \ref{SaitoCrit}, $E$ and $V$ form a basis of $T_{\mathbb{C}^2}(-\log D)$. Computing the Lie bracket, we obtain 
\[
[E,V] = (n-q-p)V.
\] Since $f$ is reduced, and assuming that $D$ is singular, we find by looking at the possible monomials in $f$ that $n - q - p \geq 0$. The case $n = p + q$ corresponds to $D$ being a normal crossing divisor, which will be considered in Section \ref{reductivefreedivisors}. We assume that $n - q - p >  0$, in which case $E$ and $V$ generate the unique $2$-dimensional non-abelian Lie algebra $\mathfrak{g}_{A}$. The logarithmic tangent bundle is then given by an action algebroid $T_{\mathbb{C}^2}(-\log D) \cong \mathfrak{g}_{A} \ltimes \mathbb{C}^2$. However, the integrating Lie groupoid $\Pi(\mathbb{C}^2, D)$ may fail to be an action groupoid, since the vector field $V$ may fail to be complete. The vector field $E$, on the other hand, is complete, and it generates an action of $\mathbb{C}^{*}$ on $\mathbb{C}^2$:
\[
\mu \ast (x,y) = (\mu^p x, \mu^q y). 
\] 
By Theorem \ref{whgconst}, $\Pi(X,D)$ defines a positive homogeneous groupoid. The derivative of the map $\pi : \Pi(X,D) \to \mathbb{C}$ can be described by the following composition of algebroid morphisms
\[
T_{\mathbb{C}^2}(-\log D) \cong \mathfrak{g}_{A} \ltimes \mathbb{C}^2 \to \mathfrak{g}_{A} \to \mathbb{C},
\]
where the first map is the projection, and the second map sends $E$ to $1$ and $V$ to $0$. 

\begin{example}[Cusp singularity] \label{cuspsinggroup}
Let $E = 3x \partial_{x} + 2y \partial_{y}$, and $f = x^2 - y^3$. Then $D$ is a curve with a cusp singularity at the origin and $X \setminus D$ is homotopic to the complement of the trefoil knot in $S^3$. In this case we find that 
\[
E(f) = 6f, \qquad V = 3y^2 \partial_{x} + 2x \partial_{y}, \qquad [E, V] = V. 
\]
Consider the map $f : \mathbb{C}^2 \to \mathbb{C}$ as a fibration, which has $D$ as its central fibre. It is a submersion away from the origin, and hence all other fibres are smooth, once-punctured, elliptic curves. Furthermore, $f$ is $\mathbb{C}^*$-equivariant with respect to the weight $6$ action of $\mathbb{C}^*$ on $\mathbb{C}$. This implies that all non-zero fibres are isomorphic, and that there is an action of the group of $6^{th}$-roots of unity on each fibre $f^{-1}(z)$ in a way that preserves the puncture. This uniquely determines the curve $f^{-1}(z)$ and it's puncture up to isomorphism. Namely it is isomorphic to $\mathbb{C}/(\mathbb{Z} + \eta \mathbb{Z})$, with the puncture at origin, where $\eta = \exp(\frac{\pi i}{3})$ is a primitive $6^{th}$ root of unity. 

The vector field $V$ is non-vanishing away from the origin in $\mathbb{C}^2$, and it points along the fibres of $f$. In other words, it restricts to a non-vanishing vector field on each punctured elliptic curve. Therefore it is not complete and hence $\Pi(X,D)$ is not an action groupoid. 

To obtain an action groupoid, consider the fibrewise compactification $F: \overline{X} \to \mathbb{C}$. The generic fibres are smooth compact elliptic curves, and the fibre above $0$, denoted $\overline{D}$, is a singular rational curve. The fibre $\overline{D}$ is a free divisor in $\overline{X}$, and the vector fields $E$ and $V$ extend to $\overline{X}$ to define a basis of $T_{\overline{X}}(-\log \overline{D})$. These vector fields are now both complete and hence the twisted fundamental groupoid is an action groupoid 
\[
\Pi(\overline{X}, \overline{D}) \cong (\mathbb{C} \ltimes \mathbb{C}) \ltimes \overline{X}. 
\]
The restriction $\Pi(\overline{X}, \overline{D})|_{X}$ is a groupoid integrating $T_{X}(-\log D)$ but it is not source simply connected, since $\pi_{1}(\overline{X} \setminus \overline{D}) \not \cong \pi_{1}(X \setminus D)$. We may obtain $\Pi(X,D)$ by taking the universal source simply connected cover of $\Pi(\overline{X}, \overline{D})|_{X}$. 
\end{example}

\subsubsection{Linear and reductive free divisors} \label{reductivefreedivisors}
In this section we recall the notions of linear and reductive free divisors which were introduced by \cite{MR2228227} and further studied in \cite{MR2521436, MR2795728, MR3237442, Linearfreeunpub}. These produce examples of positive homogeneous groupoids which are also of the form considered in Section \ref{algebraicgroupex}. 

Let $D \subset V$ be a free divisor in an $n$-dimensional complex vector space. Then $D$ is a \emph{linear} free divisor if $T_{V}(-\log D)$ admits a basis of linear vector fields. In this case, $D$ is defined by the vanishing of a homogeneous degree $n$ polynomial $f$. Let $\mathfrak{g} \subset H^{0}(V, T_{V}(-\log D))$ denote the subspace of linear logarithmic vector fields. Then $\mathfrak{g}$ is a complex $n$-dimensional Lie subalgebra of $\mathfrak{gl}(V^{*})$ and the logarithmic tangent bundle is canonically isomorphic to the corresponding action Lie algebroid 
\[
T_{V}(-\log D) \cong \mathfrak{g} \ltimes V. 
\]
Let $G' \subset GL(V)$ be the group of linear transformations which preserve $D$ and let $G$ be the connected component of the identity. Then $G$ is an algebraic subgroup of $GL(V)$ whose Lie algebra is isomorphic to $\mathfrak{g}$ \cite{MR2521436}. The divisor $D$ is \emph{reductive} if $G$ is a reductive group. 

The action groupoid $G \ltimes V$ provides an integration of $T_{V}(-\log D)$, and the twisted fundamental groupoid is given by $\Pi(V, D) = \tilde{G} \ltimes V,$ where $\tilde{G}$ is the universal cover of $G$. This is a groupoid of the type considered in Section \ref{algebraicgroupex}, and hence it gives a homogeneous groupoid. Furthermore, since $D$ is preserved by scalar multiplication, the homogeneous groupoid is positive. 

Assuming now that $G$ is reductive, its Lie algebra decomposes canonically as 
\[
\mathfrak{g} = \mathfrak{z} \oplus \mathfrak{s},
\]
where $\mathfrak{z}$ is the centre, and $\mathfrak{s}$ is the derived subalgebra, which is semisimple. Hence $\tilde{G} = \mathfrak{z} \times S$, where $S$ is a simply-connected semisimple algebraic group and $\mathfrak{z}$ is an $r$-dimensional complex vector space. The vector space $V$ is a representation of the reductive group $G$. Therefore it decomposes into irreducible subrepresentations 
\[
V = V_{1} \oplus ... \oplus V_{l}. 
\]
In \cite{MR2795728} it is shown that the irreducibles in this decomposition are pairwise non-isomorphic, making the decomposition unique. In addition, the numbers $l$ and $r$ both coincide with the number of irreducible components of $D$. Finally, the centre of $G$ is given by $(\mathbb{C}^{*})^r$, with each factor acting by scalar multiplication on the corresponding subrepresentation $V_{i}$. Therefore, fixing an ordering of the $V_{i}$, the Lie algebra of the centre, $\mathfrak{z}$, is canonically isomorphic to $\mathbb{C}^r$, and the kernel of the exponential map $\exp: \mathfrak{z} \to (\mathbb{C}^{*})^r$ is a lattice, which is canonically isomorphic to $\mathbb{Z}^r$. This lattice coincides with the elements of $\mathfrak{z} \subseteq \tilde{G}$ that act trivially on $V$.  

There is an abundance of examples of linear free divisors, and the work of \cite{MR2228227, MR2521436, MR430336} provides several constructions. In particular, \cite{MR2228227} explains how to construct reductive free divisors as discriminants in the representation spaces of quivers. We will not review these constructions in this paper, but rather we will present some of the resulting examples, using the quivers to label the divisors. 

\begin{example}[$A_{n+1}$] \label{Aquiver}
Let $X = \mathbb{C}^{n}$ and let $D$ be a simple normal crossing divisor defined by the vanishing of the product of the coordinate functions $x_{1}x_{2}...x_{n}$. The reductive group of symmetries is $G = (\mathbb{C}^*)^{n}$, and the twisted fundamental groupoid is given by $\Pi(X,D) = \mathbb{C}^{n} \ltimes \mathbb{C}^n$. These divisors are associated to the $A_{n+1}$ quivers. 
\end{example}

\begin{example}[$D_{4}$] \label{D4quiver}
Consider the reductive group $G = (\mathbb{C}^*)^{3} \times SL(2, \mathbb{C})$ acting linearly on $X = (\mathbb{C}^{2})^{\oplus 3}$ in the following way: 
\[
(a, b, c, M) \ast (u, v, w) = (aM(u), bM(v), cM(w)),
\]
where $a, b, c \in \mathbb{C}^*$, $M \in SL(2, \mathbb{C})$, and $u, v, w \in \mathbb{C}^2$. This action has an open dense orbit $U$ whose complement $D = X \setminus U$ is a reductive linear free divisor defined by the vanishing of the following polynomial: 
\[
f(u,v,w) = (u_{1}v_{2} - u_{2}v_{1})(v_{1}w_{2} - v_{2}w_{1})(w_{1}u_{2} - w_{2}u_{1}). 
\]
The twisted fundamental groupoid is given by 
\[
\Pi(X,D) =  (\mathbb{C}^{3} \times SL(2, \mathbb{C})) \ltimes (\mathbb{C}^{2})^{\oplus 3}. 
\]
This divisor is associated to the $D_{4}$ quiver. 
\end{example}

\begin{example}[$G_{2}$] \label{G2example}
Consider the group $GL(2, \mathbb{C})$ and its canonical representation on $\mathbb{C}^2$. By taking the symmetric product, we obtain a representation on $X = S^3(\mathbb{C}^2) \cong \mathbb{C}^4$. This representation has an open dense orbit $U$ whose complement $D = X \setminus U$ is a reductive linear free divisor. 

Let $e_{1}$ and $e_{2}$ be the standard basis vectors of $\mathbb{C}^2$, and let $(x,y,z,w)$ be coordinates on $X$ which denote the points
\[
xe_{1}^{3} + ye_{1}^2e_{2} + ze_{1}e_{2}^2 + we_{2}^3. 
\]
In these coordinates, the divisor $D$ is defined by the vanishing of the following polynomial 
\[
f(x,y,z,w) = 27w^2x^2 - 18 wxyz + 4wy^3 + 4xz^3 - y^2z^2.
\]
If we identify $X$ with the space of homogeneous degree $3$ polynomials in two variables, then $D$ corresponds to the discriminant, where the polynomials have repeated roots. 

The twisted fundamental groupoid is given by 
\[
\Pi(X,D) = (\mathbb{C} \times SL(2,\mathbb{C})) \ltimes S^3(\mathbb{C}^2),
\]
where we use the fact that $\mathbb{C} \times SL(2,\mathbb{C})$ is the universal cover of $GL(2, \mathbb{C})$. 

We can think of this divisor as being associated to the $G_{2}$ quiver. This quiver is a quotient of the $D_{4}$ quiver, and there is a corresponding morphism between twisted fundamental groupoids: 
\begin{align*}
(\mathbb{C}^{3} \times SL(2, \mathbb{C})) \ltimes (\mathbb{C}^{2})^{\oplus 3} &\to (\mathbb{C} \times SL(2,\mathbb{C})) \ltimes S^3(\mathbb{C}^2) \\
(x,y,z,M, u, v, w) & \mapsto (\frac{1}{3}(x + y + z), M, uvw). 
\end{align*}
\end{example}

\begin{example} \label{symmetricborellinearfree}
The following family of examples, described in detail in \cite{MR2521436}, are linear but non-reductive free divisors. Let $B_{n}$ be the group of upper triangular invertible $n \times n$ matrices, and let $S_{n}$ be the vector space of symmetric $n \times n$ matrices. There is a right linear action of $B_{n}$ on $S_{n}$ given by 
\[
S \ast B = B^{T}SB,
\]
where $B \in B_{n}, S \in S_{n}$ and $B^{T}$ denotes the transpose of $B$. This action has an open dense orbit $U$, whose complement $D_{n} = S_{n} \setminus U$ is a linear free divisor. Let $det_{j} : S_{n} \to \mathbb{C}$ denote the upper left $j \times j$ minor. Then $D_{n}$ is defined by the vanishing of the following polynomial 
\[
f_{n} = \prod_{j = 1}^n det_{j},
\]
and the twisted fundamental groupoid is given by $\Pi(S_{n}, D_{n}) = S_{n} \rtimes \tilde{B}_{n}$, where $\tilde{B}_{n}$ is the universal cover of $B_{n}$. 
When $n = 2$, we have $S_{2} \cong \mathbb{C}^3$, and $D_{2}$ given by the vanishing of 
\[
f_{2} = x (y^2 - xz), 
\]
which is the union of a quadric cone and one of its tangent planes. 
\end{example}

\subsubsection{Sekiguchi's 17 free divisors} \label{SekiDiv}
In \cite{MR2588504} Sekiguchi describes and classifies a family of $17$ weighted homogeneous free divisors in $\mathbb{C}^3$. They are divided into $3$ subfamilies, $A$, $B$, and $H$, according to the weights of the $\mathbb{C}^*$-action with respect to which they are weighted homogeneous. The weights of the coordinates $(x,y,z)$ for each subfamily are respectively $(2, 3, 4)$, $(1,2,3)$ and $(1, 3, 5)$. We refer to \cite{MR2588504} for the full family of divisors. In this section we give one example from subfamily $B$. It is a divisor $D$ defined by the following polynomial 
\[
F_{B,5} = xy^4 + y^3 z + z^3,
\]
which is weighted homogeneous of degree $9$ with respect to the Euler vector field 
\[
E = x \partial_{x} + 2y \partial_{y} + 3z \partial_{z}.
\]
The logarithmic tangent bundle $T_{\mathbb{C}^3}(-\log D)$ has a basis given by the vector fields $E$ and 
\begin{align*}
V &= 2y \partial_{x} + (-24 xy + 2z) \partial_{y} + (-2y^2 - 32 xz) \partial_{z} \\ 
W &= 3z \partial_{x} - 9y^2 \partial_{y} - 12yz \partial_{z}. 
\end{align*}
These vector fields satisfy $[E, V] = V$, $[E, W] = 2W$ and $[V,W] = 24zE + 6yV - 40xW$. Furthermore, they satisfy 
\[
E(F_{B,5}) = 9F_{B,5}, \ \ V(F_{B,5}) = -96xF_{B,5}, \ \ W(F_{B,5}) = -36yF_{B,5}.
\]
The isotropy Lie algebra at the origin is a $3$-dimensional Lie algebra spanned by $E_{0}, V_{0}, W_{0}$, with brackets $[E_{0}, V_{0}] = V_{0}$, $[E_{0}, W_{0}] = 2W_{0}$ and $[V_{0},W_{0}] = 0$. However, unlike the previous Lie algebroids we have been considering, $T_{\mathbb{C}^3}(-\log D)$ is not an action algebroid. This can be seen as a consequence of our classification of representations, and is given in Corollary \ref{notaction}.

\section{Representation theory of homogeneous groupoids} \label{ReptheoryWHG}
In this section, we develop the main structure theorems for the representation theory of homogeneous Lie groupoids. 

\subsection{Morita equivalence}\label{MEsection}
To start, we briefly recall the notion of Morita equivalence, which is a powerful tool in the study of groupoid representations. A \emph{Morita equivalence} between a Lie groupoid $\mathcal{G}$ over $X$ and a Lie groupoid $\mathcal{H}$ over $Y$ is a bi-principal $(\mathcal{G}, \mathcal{H})$ bi-bundle. Morita equivalence defines an equivalence relation between Lie groupoids that is much weaker than isomorphism. This makes it particularly useful because \emph{Morita equivalent Lie groupoids have equivalent categories of representations}. As an application, we have the following useful result, which allows us to go between local and global descriptions of Lie algebroid representations. 
\begin{lemma} \label{restrictionprop}
Let $V$ be a representation of $\mathbb{C}^*$ with only positive weights, and let $\mathbb{C} \ltimes V$ be the action groupoid obtained by letting $\mathbb{C}$ act on $V$ via exponentiation. Let $\mathcal{G}$ be a source simply connected Lie groupoid over $V$ which is equipped with a morphism $i: \mathbb{C} \ltimes V \to \mathcal{G}$ (in particular, $\mathcal{G}$ could be a positive homogeneous groupoid), and let $A$ be the Lie algebroid of $\mathcal{G}$. Finally, let $B \subset V$ be a neighbourhood of the origin which is closed under the action of elements $z \in \mathbb{C}^*$ with $|z| \leq 1$. Then by restricting representations of $A$ to $B$, we obtain the following equivalence of categories 
\[
Rep(A, H) \cong Rep(A|_{B}, H). 
\]
\end{lemma} 
\begin{proof}
Because the action of $\mathbb{C}^*$ on $V$ has only positive weights, the neighbourhood $B$ intersects every orbit of $\mathcal{G}$. By \cite[Proposition 4.5]{bischoff2020lie} this implies that $\mathcal{G}$ and $\mathcal{G}|_{B}$ are Morita equivalent, and hence they have equivalent categories of representations. Since $\mathcal{G}$ is source simply connected, its category of representations is equivalent to $Rep(A, H)$ by Lie's second theorem. The Lie algebroid of $\mathcal{G}|_{B}$ is $A|_{B}$. Hence, the proof follows if we can show that $\mathcal{G}|_{B}$ is source simply connected, since we can then apply Lie's second theorem to conclude that $Rep(\mathcal{G}|_{B}, H) \cong Rep(A|_{B}, H).$ To see that the source fibres of $\mathcal{G}|_{B}$ are simply connected, observe that the $\mathbb{C}^*$-action allows us to deform paths and homotopies from $\mathcal{G}$ so that they lie entirely in $\mathcal{G}|_{B}$. 
\end{proof}

As an upshot of Lemma \ref{restrictionprop}, we note the following corollary, which makes precise the fact that we can often study connections by working purely locally with germs. Let $\mathcal{G}$ be a groupoid satisfying the properties described in Lemma \ref{restrictionprop}, and let $A$ be its Lie algebroid. We may define a category of germs of $H$-representations $Rep_{0}(A, H)$. The objects of this category consist of triples $(B, P, \nabla)$, where $0 \in B \subseteq V$ is an open neighbourhood of the origin, $P \to B$ is a principal $H$-bundle, and $\nabla: A|_{B} \to At(P)$ is a representation. The set of morphisms between objects $(B, P, \nabla)$ and $(B', P', \nabla')$ consist of equivalence classes of tuples $(C, T)$, where $0 \in C \subseteq B \cap B'$ is an open neighbourhood, and $T: \nabla|_{C} \to \nabla'|_{C}$ is a morphism. Two tuples $(C_{1}, T_{1})$ and $(C_{2}, T_{2})$ are equivalent if there is an open subset $0 \in C_{12} \subseteq C_{1} \cap C_{2}$ such that $T_{1}|_{C_{12}} = T_{2}|_{C_{12}}$. 

\begin{corollary}
Let $\mathcal{G}$ be a groupoid satisfying the properties described in Lemma \ref{restrictionprop}. Then there is an equivalence of categories 
\[
R: Rep(A, H) \to Rep_{0}(A, H).
\]
It is given by sending a representation $(P, \nabla)$ to $(V, P, \nabla)$, and by sending a morphism $T: \nabla \to \nabla'$ to the equivalence class of $(V, T)$. 
\end{corollary}

\subsection{Jordan-Chevalley decomposition}\label{JCdecompsection}
In this section, we establish a Jordan-Chevalley decomposition theorem for representations of homogeneous groupoids. When this is applied to the groupoid $\mathbb{C} \ltimes \mathbb{C}$ (i.e. the groupoid associated to the $A_{2}$ quiver in Example \ref{Aquiver}), this result recovers the well-established Jordan decomposition theorem for differential equations with Fuchsian singularities due to Hukuhara, Turrittin and Levelt \cite{MR5229, turrittin1955convergent, levelt1975jordan} (this result extends to formal differential equations with irregular singularities; see also \cite{wasow2018asymptotic, MalgrangeReduction, MR610528, babbitt1983formal, MR4186767}). This result was reinterpreted from the perspective of Lie groupoids in Theorem 3.9 of \cite{bischoff2020lie}. It is this later version of the result which we will generalize in the present section. Throughout this section, we assume that $(\mathcal{G}, i, \pi)$ is a homogeneous groupoid over $X$. 

We begin by recalling the multiplicative Jordan-Chevalley decomposition for elements of a complex reductive group $H$. An element $g \in H$ is defined to be \emph{semisimple} or \emph{unipotent} if its image under a faithful representation of $H$ is, respectively, semisimple (i.e. diagonalizable) or unipotent (i.e. $(g-id)^n = 0$  for some $n \in \mathbb{N}$). An arbitrary element $g \in H$ admits a unique Jordan-Chevalley decomposition $g = su$, where $s$ is semisimple, $u$ is unipotent, and $s$ and $u$ commute. This decomposition is preserved by morphisms of reductive groups, and it is compatible via the exponential map with the additive Jordan-Chevalley decomposition for elements of the Lie algebra of $H$. 

\subsubsection*{Monodromy automorphism}
Let $P \to X$ be a principal $H$-bundle equipped with a representation of $\mathcal{G}$
\[
\phi: \mathcal{G} \to \mathcal{G}(P). 
\]
By pulling back $\phi$ via the map $i \circ j : \mathbb{Z}^k \times X \to \mathcal{G}$, we define the \emph{monodromy automorphism} 
\[
M_{n}(x) = (\phi \circ i \circ j) (n, x) \in Aut_{H}(P_{x}),
\]
for $n \in \mathbb{Z}^k$. Because $\mathbb{Z}^k \times X$ is central, we have 
\[
M_{n}(t(g)) \circ \phi(g) = \phi(g) \circ M_{n}(s(g)),
\]
for all $g \in \mathcal{G}$, and so $M_{n}$ actually defines an automorphism of $\phi$. We therefore obtain a homomorphism 
\[
M : \mathbb{Z}^k \to Aut(\phi). 
\]
Given a (local) trivialization of $P$, we may view $M_{n}(x) \in H$. Because changing trivializations has the effect of conjugating $M_{n}(x)$, it makes sense to say that $M_{n}(x)$ lies in a fixed conjugacy class of $H$. The same is true for the semisimple and unipotent components of $M_{n}(x)$, which are defined by the Jordan-Chevalley decomposition.

\begin{lemma} \label{JCdeclemma}
Let $S_{n}$ and $U_{n}$ denote the semisimple and unipotent components of the monodromy $M_{n}$ respectively, defined using the multiplicative Jordan-Chevalley decomposition. Then $S_{n}$ and $U_{n}$ define holomorphic automorphisms of $\phi$. Furthermore, the conjugacy class of $S_{n}$ is constant over $X$.  
\end{lemma}
\begin{proof}
Choose a faithful representation $V$ of $H$, which allows us to view $H \subseteq GL(V)$. This determines an associated vector bundle $V_{P} = (P \times V)/H$, and an embedding of the gauge groupoid
\[
e: Aut_{H}(P) \to End(V_{P}). 
\]
Because morphisms of reductive groups preserve the Jordan-Chevalley decomposition, $e(S_{n}(x))$ is the semisimple component of $e(M_{n}(x))$.

The conjugacy class of $M_{n}(x)$ (and hence of $e(M_{n}(x))$) is constant along the orbits of $\mathcal{G}$. Indeed, let $g \in \mathcal{G}$ be an element with $s(g) = y$ and $t(g) = x$. Then $M_{n}(x) = \phi(g)M_{n}(y)\phi(g)^{-1}$. Furthermore, since an arbitrary element of $\mathcal{G}$ can be written as a product of elements that are close to the identity, we may assume that $x$ and $y$ lie in the common domain of a trivialization of $P$. This allows us to view $\phi(g)$ as an element of $H$. 

As a result, the characteristic polynomial (or spectrum) of $e(M_{n}(x))$, defined as 
\[
c(\lambda, x) = det( \lambda id - e(M_{n}(x))),
\]
is constant along the orbits. By continuity $c(\lambda, x)$ is constant along the orbit closures, and hence along the s-equivalence classes. Since $X$ has a unique s-equivalence class, it follows that $c(\lambda, x)$ is constant along $X$. Hence, we have a single characteristic polynomial, and we factor it as follows: 
\[
c(\lambda) = \prod_{i} (\lambda - \mu_{i})^{d_{i}}, 
\]
where $\mu_{i}$ are the distinct eigenvalues. The rank of the operator $(e(M_{n}) - \mu_{i})^{d_{i}}$ is constant, and its kernel $V_{i}$, which is the generalized eigenspace for $\mu_{i}$, is a subbundle of $V_{P}$ of dimension $d_{i}$. We thereby obtain the generalized eigenspace decomposition 
\[
V_{P} = \bigoplus_{i} V_{i}. 
\]
Using the Chinese remainder theorem, we can find a polynomial $p(t)$ such that 
\[
p(t) = 0 \ mod \ t, \qquad p(t) = \mu_{i} \ mod \ (t - \mu_{i})^{d_{i}}.
\]
Let $T = p(e(M_{n})) \in End(V_{P}). $ Then $T|_{V_{i}} = \mu_{i}$, and hence $T$ is semisimple. (This argument is recalled from \cite{humphreys2012introduction}.) This implies that $T = e(S_{n})$, which in turn implies that $S_{n}$ varies holomorphically over $X$. Since $U_{n} = S_{n}^{-1}M_{n}$, this also implies that $U_{n}$ varies holomorphically over $X$. 

The monodromy $M_{n}$ is an automorphism of $\phi$, which means that it is invariant under the $\mathcal{G}$ action induced by $\phi$. Since the Jordan-Chevalley decomposition is preserved by morphisms of reductive groups, it follows that $S_{n}$ and $U_{n}$ are also invariant. In other words, they are automorphisms of $\phi$. 

Finally, by the same argument as above, the conjugacy class of $S_{n}$ is constant over the orbits of $\mathcal{G}$. But $S_{n}$ is semisimple and so its conjugacy class is closed. Therefore, the conjugacy class of $S_{n}$ is constant over the orbit closures, and therefore over the unique s-equivalence class. 
\end{proof}

One upshot of Lemma \ref{JCdeclemma} is that the unipotent monodromy defines a homomorphism 
\[
U : \mathbb{Z}^k \to Aut(\phi),
\]
whose image consists of unipotent automorphisms.

\subsubsection*{Groupoid 1-cocycles}
We now recall the notion of groupoid $1$-cocycles valued in a representation $(P, \phi)$. These are the structures that allow us to `deform' $\phi$ to a different representation $\psi$ defined on the same bundle $P$. More precisely, a \emph{$1$-cocycle} for a groupoid $\mathcal{G}$, valued in a representation $(P, \phi)$, is a holomorphic section $\sigma$ of the bundle of groups $t^{*}Aut_{H}(P) \to \mathcal{G}$ satisfying the following cocycle condition: 
\begin{equation} \label{cocycleeq}
\sigma(gh) = \sigma(g) \phi(g) \sigma(h) \phi(g)^{-1},
\end{equation}
for all composable pairs of arrows $(g, h) \in \mathcal{G}^2$. It is straightforward to check that, given such a cocycle $\sigma$, the composite $\psi = \sigma \circ \phi$ defines a new representation on the bundle $P$. 

We will make use of a general construction for producing $1$-cocycles out of automorphisms. 

\begin{proposition} \label{generalcocycleconst}
Suppose we have the data of a Lie groupoid $\mathcal{G}$ over $X$, a Lie group $K$, and a homomorphism $f: \mathcal{G} \to K$. Given a representation $(P, \phi)$, suppose that we have a homomorphism $c: K \to Aut(\phi)$. Then there is a $1$-cocycle $\sigma$, defined as follows: 
\[
\sigma: \mathcal{G} \to t^{*}Aut_{H}(P), \qquad g \mapsto c(f(g))|_{t(g)}. 
\]
\end{proposition}
\begin{proof}
We need to verify Equation \ref{cocycleeq} for $\sigma$. Note first that the map $c$ satisfies, for all $k \in K$ and $g \in \mathcal{G}$, the following equation
\[
c(k)|_{t(g)} \phi(g) = \phi(g) c(k)|_{s(g)}. 
\]
Using the properties of $c$ and $f$, and the fact that for a composable pair $(g,h)$ we have $t(h) = s(g)$, we compute: 
\begin{align*}
\sigma(g) \phi(g) \sigma(h) \phi(g)^{-1} &= c(f(g))|_{t(g)} \phi(g) c(f(h))|_{t(h)} \phi(g)^{-1} \\
&= c(f(g))|_{t(g)} \phi(g) c(f(h))|_{s(g)} \phi(g)^{-1} \\
&= c(f(g))|_{t(g)}  c(f(h))|_{t(g)} \\
&= c(f(gh))|_{t(gh)} \\
&= \sigma(gh). 
\end{align*}
\end{proof}

\subsubsection*{Jordan-Chevalley decomposition}
Returning to the setting of an $H$-representation $(P, \phi)$ of a homogeneous groupoid $(\mathcal{G}, i, \pi)$, we observe that we are very close to the set-up of Proposition \ref{generalcocycleconst}. Indeed, the structure of a homogeneous groupoid includes the data of a groupoid morphism 
\[
\pi : \mathcal{G} \to \mathbb{C}^k,
\]
and, by Lemma \ref{JCdeclemma}, the (inverse of the) unipotent monodromy gives rise to a homormophism 
\[
U^{-1} : \mathbb{Z}^k \to Aut(\phi). 
\]
We therefore need to extend $U^{-1}$ to a map out of $\mathbb{C}^k$. We do this by using the logarithm. 

Recall that the logarithm is well-defined for the unipotent elements of a reductive group (for example, by using the Taylor series expansion in any faithful representation). The automorphism $U_{m}$ is a unipotent section of $Aut_{H}(P)$, which is a bundle of reductive groups isomorphic to $H$. Thus we may consider $\log(U_{m})$, which is a nilpotent section of $\mathfrak{aut}_{H}(P)$. To see that this varies holomorphically, we trivialize $P$ over an open neighbourhood $W \subseteq X$, and choose a faithful representation $V$ of $H$. This allows us to view $U_{m} : W \to GL(V)$. The logarithm on unipotent matrices is given by a polynomial, and so $\log(U_{m})$ is holomorphic. This perspective also lets us see that $\log(U_{n})$ and $\log(U_{m})$ commute at each point, so that their sum is nilpotent, and 
\[
\exp(\log(U_{n}) + \log(U_{m})) = U_{n+m}.
\]
Hence $\log(U_{n}) + \log(U_{m}) = \log(U_{n+m})$, and so we get a homomorphism $\log(U) : \mathbb{Z}^k \to H^{0}(X, \mathfrak{aut}_{H}(P))$. 

Recall from Section \ref{Firstsection} that an automorphism of $\phi$, such as $U_{m}$, is a gauge transformation which is invariant under the conjugation action of $\mathcal{G}$ on $Aut_{H}(P)$ induced by $\phi$. Because the exponential (and hence the logarithm) intertwines the actions of $\mathcal{G}$ on $Aut_{H}(P)$ and $\mathfrak{aut}_{H}(P)$, it follows that $\log(U_{m})$ is an infinitesimal automorphism of $\phi$. Hence we have a homomorphism 
\[
\log(U) : \mathbb{Z}^{k} \to \mathfrak{aut}(\phi). 
\]
This map is additive, and so it can be extended to a linear map $\mathbb{C}^{k} \to \mathfrak{aut}(\phi)$. In fact, this is a map of Lie algebras since the $\log(U_{m})$ are mutually commuting endomorphisms. 

This allows us to define a homomorphism 
\[
c : \mathbb{C}^k \to Aut(\phi), \qquad z \mapsto \exp(\frac{-1}{2 \pi i } \sum_{i} z_{i} \log U_{e_{i}}),
\]
where $e_{i}$ form the standard basis of $\mathbb{Z}^k$, and $z = \sum_{i} z_{i} e_{i}$. It is straightforward to check that this homomorphism extends $U^{-1}$. Applying Proposition \ref{generalcocycleconst} we get a $1$-cocycle 
\[
\sigma_{U} : \mathcal{G} \to t^*Aut_{H}(P), \qquad g \mapsto c(\pi(g))|_{t(g)} = \exp( \frac{-1}{2 \pi i } \sum_{i} \pi(g)_{i} \log U_{e_{i}}(t(g)) ).
\]
\begin{corollary}
The section $\sigma_{U}$ is a $1$-cocycle valued in the representation $(P, \phi)$. The resulting representation $\sigma_{U} \phi$ has semisimple monodromy given by $S$, the semisimple component of the monodromy of $\phi$. 
\end{corollary}

We encode the results of the discussion in terms of an equivalence of categories. Let $\mathcal{JC}_{H}$ denote the category of tuples $(P, \phi_{s}, U)$, where $\phi_{s}$ is a representation of $\mathcal{G}$ on a principal $H$-bundle $P$ which has semisimple monodromy, and $U : \mathbb{Z}^k \to Aut(\phi_{s})$ is a homomorphism whose image consists of unipotent automorphisms. The morphisms of the category are given by principal $H$-bundle maps that intertwine both the representations and the unipotent automorphisms. 

\begin{theorem} \label{JCdecomposition}
There is an isomorphism of categories 
\[
\Rep(\mathcal{G}, H) \cong \mathcal{JC}_{H}. 
\]
\end{theorem}
\begin{proof}
We define the functor $T : \Rep(\mathcal{G}, H) \to \mathcal{JC}_{H}$ as follows. Given a representation $(P, \phi)$, let $U : \mathbb{Z}^k \to Aut(\phi)$ be the unipotent component of the monodromy of $\phi$. Then let $\sigma_{U}$ denote the corresponding $1$-cocyle, and let $\phi_{s} = \sigma_{U} \phi$. Then $\phi_{s}$ has semisimple monodromy, and $U$ also gives automorphisms of $\phi_{s}$. Hence, we define 
\[
T(P, \phi) = (P, \phi_{s}, U). 
\]
The functor acts as the identity on morphisms. Conversely, let $(P, \phi_{s}, U)$ be an object of $\mathcal{JC}_{H}$. Then define 
\[
T^{-1}(P, \phi_{s}, U) = (P, \sigma_{U^{-1}}\phi_{s}). 
\]
The two operations are inverses to each other, and hence we obtain the desired isomorphism. 
\end{proof}

\subsection{Linearization}  \label{LinearizationSection}
In this section, we establish a linearization theorem that generalizes \cite[Proposition 3.6 ]{bischoff2020lie}. This result leads to partial linearization theorems for representations of positive homogeneous groupoids.

Recall that up to a finite cover, a connected complex reductive group $G$ has the form 
\[
G = (\mathbb{C}^*)^{k} \times S,
\]
where $S$ is a simply connected semisimple group. Denote the universal cover $\tilde{G} = \mathbb{C}^k \times S$. Let $V$ be a representation of $G$, and assume that $(\mathbb{C}^*)^{k}$ contains a $\mathbb{C}^*$ factor whose action on $V$ has only positive weights. We are interested in studying the representations of 
\[
\mathcal{G} = (\mathbb{C}^{k} \times S) \ltimes V,
\]
which is a positive homogeneous groupoid of the form considered in Section \ref{algebraicgroupex}. 

To begin, consider the following groupoid homomorphisms 
\begin{align*}
p: \tilde{G} \ltimes V \to \tilde{G}, \qquad (g,v) \mapsto g \\
\iota:  \tilde{G} \to \tilde{G} \ltimes V, \qquad g \mapsto (g,0),
\end{align*}
which satisfy $p \circ \iota = id_{\tilde{G}}$. By pulling back representations we obtain the following two functors 
\begin{align*}
p^* :  \Rep(\tilde{G}, H) \to \Rep(\tilde{G} \ltimes V, H), \\
\iota^* : \Rep(\tilde{G} \ltimes V, H) \to \Rep(\tilde{G}, H),
\end{align*}
which satisfy $\iota^{*} \circ p^{*} = id$. We say that representations in the image of $p^*$ are \emph{trivial}. Using these functors we define a \emph{linear approximation} functor which sends any given representation to an associated trivial one:
\[
L = p^{*} \circ \iota^* : \Rep(\tilde{G} \ltimes V, H) \to \Rep(\tilde{G} \ltimes V, H). 
\]
\begin{definition} \label{linearizationdefinition}
Let $\phi$  be an $H$-representation of $\tilde{G} \ltimes V$. A \emph{linearization} of $\phi$ is an isomorphism of representations
\[
T : L(\phi) \to \phi. 
\]
\end{definition}
Linearizations are not guaranteed to exist. To see this, observe that the monodromy of a trivial representation is constant, and therefore, the monodromy of a linearizable representation lies in a fixed conjugacy class. Hence, the monodromy of a representation may obstruct the existence of linearizations. The following result gives a criterion, in terms of the monodromy, for the existence of linearizations. 

\begin{theorem}\label{mainlinearizationthm}
Let $\phi$ be an $H$-representation of $\tilde{G} \ltimes V$ with semisimple monodromy. Then $\phi$ can be linearized. 
\end{theorem}
\begin{proof}
In order to simplify the proof, we trivialize the principal bundle $P$ underlying the representation, so that we may view $\phi$ as a homomorphism $\phi: \tilde{G} \ltimes V \to H$. Let $F = \iota^*(\phi)$, so that $L(\phi)(g,v) = F(g)$. The goal is to construct an isomorphism between $L(\phi)$ and $\phi$. Our strategy in constructing the isomorphism will be to reduce ourselves to the situation of a compact group action, and then to apply a Bochner linearization argument. 

By Lemma \ref{restrictionprop}, it suffices to construct the linearization isomorphism on an arbitrarily small neighbourhood of the origin, and so we may work locally. In what follows, we always choose neighbourhoods of $0$ which satisfy the hypothesis of Lemma \ref{restrictionprop}. This is always possible, for example, by choosing polydiscs. Given such a neighbourhood $B$, we denote the restriction of the groupoid by $\mathcal{G}_{B}$. In the proof of Lemma \ref{restrictionprop} we showed that this groupoid is source simply connected. 

We will use the notation $M_{n}^{\phi}$ to denote the monodromy of a representation $\phi$, for $n \in \mathbb{Z}^k$. Let $e_{1}, ..., e_{k}$ denote the standard basis of $\mathbb{Z}^k$, and let $M_{i} = M_{e_{i}}^{\phi}$. Our assumption is that $M_{i}(v) \in H$ is semisimple for all $i$ and $v \in V$. By Lemma \ref{JCdeclemma}, each $M_{i}$ lies in a single closed conjugacy class $\mathcal{C}_{i}$ of $H$. 

The restriction $F|_{\mathbb{C}^k} : \mathbb{C}^k \to H$ is a Lie group homomorphism, and so has the form 
\[
F( \sum_{i} t_{i} e_{i}) = \exp( \frac{1}{2 \pi i } \sum_{i} t_{i} S_{i}),
\]
where $S_{i} \in \mathfrak{h}$ are mutually commuting elements of the Lie algebra. The monodromy of $L(\phi)$ is constant over $V$ and is given by 
\[
M_{e_{i}}^{L(\phi)}(v) = M_{i}(0) = \exp(S_{i}). 
\]
This implies in particular that $S_{i}$ are semisimple. 

We now construct a gauge transformation in a neighourhood of the origin of $V$ which takes $\phi$ to a representation with constant monodromy given by $\exp(S_{i})$. We build this gauge transformation inductively in $k$ steps:

\textbf{Step j.} Assume that the output of step $j-1$ is a holomorphic map $h_{j-1} : B_{j-1} \to H$, defined on a neighbourhood $B_{j-1} \subseteq V$ of the origin, and satisfying the following properties 
\begin{itemize}
\item $h_{j-1}(0) = 1$, 
\item $h_{j-1}(v) \exp(S_{i}) h_{j-1}(v)^{-1} = M_{i}(v)$ for all $v \in B_{j-1}$ and $i = 1, ..., j-1.$ 
\end{itemize}
Now define 
\[
\phi_{j-1} : \mathcal{G}_{B_{j-1}} \to H, \qquad (g,v) \mapsto h_{j-1}(g(v))^{-1} \phi(g,v) h_{j-1}(v).
\]
It follows from the properties of $h_{j-1}$ that $\phi_{j-1}(g,0) = F(g)$, and $M_{e_{i}}^{\phi_{j-1}}(v) = \exp(S_{i})$, for $i \leq j-1$. Since $\exp(S_{i})$, for $i = 1, ..., j-1$, are automorphisms of $\phi_{j-1}$, it follows that this representation is valued in the centralizer $C_{H}(\exp(S_{1}), ..., \exp(S_{j-1}))$. Let $C_{j-1}$ denote the connected component of the identity, which is a connected reductive group. Then because $\mathcal{G}_{B_{j-1}} $ is source connected, the image of $\phi_{j-1} $ is contained in $C_{j-1}$: 
\[
\phi_{j-1} : \mathcal{G}_{B_{j-1}} \to C_{j-1}. 
\]
Now consider the monodromy $M_{e_{j}}^{\phi_{j-1}}: B_{j-1} \to C_{j-1}$, which satisfies $M_{e_{j}}^{\phi_{j-1}}(0) = \exp(S_{j})$. It is semisimple in $C_{j-1}$ since it is semisimple in $H$ and the Jordan decomposition is preserved by morphisms of reductive groups. By Lemma \ref{JCdeclemma}, its image is contained in a single closed conjugacy class $\mathcal{C}_{j} \subseteq C_{j-1}$. It is therefore possible to find a neighbourhood $B_{j} \subseteq B_{j-1}$ and a holomorphic map $k_{j} : B_{j} \to C_{j-1}$ such that $k_{j}(0) = 1$ and 
\[
k_{j}(v) \exp(S_{j})k_{j}(v)^{-1} = M_{e_{j}}^{\phi_{j-1}}(v),
\]
for all $v \in B_{j}$. Now define $h_{j} = h_{j-1}k_{j}$, which is the output of Step $j$. 

Let $B = B_{k}$, $h = h_{k} : B \to H$, $C = C_{H}(\exp(S_{1}), ..., \exp(S_{k}))^{o}$, and 
\[
\psi = \phi_{k} : \mathcal{G}_{B} \to C, 
\]
which is the output of the $k^{th}$-step. The homomorphism $\psi$ satisfies $\psi(g,0) = F(g)$, and $M_{e_{i}}^{\psi} = \exp(S_{i})$ for all $i$. Now define the following action of $\mathcal{G}_{B}$ on $C \times B$: 
\[
(g, v) \ast (c,v) = (\psi(g, v) c F(g)^{-1}, g(v)). 
\]
This action satisfies: 
\[
(2 \pi i e_{i}, v) \ast (c,v) = (\exp(S_{i}) c \exp(S_{i})^{-1}, v) = (c,v).
\]
More generally, the subgroupoid $2 \pi i \mathbb{Z}^k \times B$ acts trivially. Hence, the action descends to an action of $(G \ltimes V)|_{B}$. The point $(1,0)$ is fixed by the action since $(g, 0) \ast (1,0) = (\psi(g,0)F(g)^{-1}, 0) = (1,0)$. Finally, the fibre $C \times \{ 0 \}$ is fixed by the action. Taking the linear approximation gives a representation of $G$ on the tangent space $T_{(1,0)}(C \times V)$ and it sits in the following short exact sequence of representations 
\[
0 \to Lie(C) \to T_{(1,0)}(C \times V) \to V \to 0. 
\]
Since $G$ is reductive, there is an equivariant splitting of this sequence $\sigma: V \to T_{(1,0)}(C \times V)$. 

By averaging over a maximal compact subgroup of $G$, we may construct a local holomorphic $G$-equivariant isomorphism between a neighbourhood of $(1,0)$ in $C \times V$ and a neighbourhood of the origin in $T_{(1,0)}(C \times V)$. Furthermore, this isomorphism can be constructed so that it is compatible with the projection to $V$. By transporting the splitting $\sigma$ across this isomorphism, we obtain a section $u: B \to C \times H$ which is $(G \ltimes V)|_{B}$-invariant (where $B$ may have to be shrunk further). In other words, the map $u$ satisfies 
\begin{equation*}
\psi(g, v) = u(g(v)) F(g) u(v)^{-1}. 
\end{equation*}\
The product $hu$ is the desired isomorphism between $L(\phi)$ and $\phi$. 
\end{proof}
\begin{remark}
A representation of $\tilde{G} \ltimes V$ on a principal bundle $P$ may be viewed as a representation of $\tilde{G}$ by automorphisms of the total space of $P$. In this way, Theorem \ref{mainlinearizationthm} is similar to results of Cartan \cite{Cartanpaper}, Bochner \cite{MR13161}, Hermann \cite{MR217225}, Kushnirenko \cite{MR0210833},  and Guillemin-Sternberg \cite{MR217226}. 
\end{remark}

A key step in the proof of Theorem \ref{mainlinearizationthm} is the construction of a gauge transformation which converts a representation into one with constant monodromy. When the group is $G = (\mathbb{C}^*)^{k} \times S$, there are $k$ independent monodromies. As a result, this part of the proof proceeds in $k$ inductive steps. When $k > 1$, it is crucial for the induction that the centralizer subgroup of the monodromies remains reductive, and this relies on the fact that the monodromies are semisimple. On the other hand, if $k = 1$ then we can proceed with the construction of the gauge transformation, as long as we assume that the monodromy remains in a single conjugacy class over $V$. The remainder of the proof then applies without modification. We therefore obtain the following strengthening of Theorem \ref{mainlinearizationthm} in the case $k = 1$. It remains open whether this strengthening is also valid for $k > 1$. 

\begin{corollary}
Suppose $G = \mathbb{C}^* \times S$. Then an $H$-representation of $\tilde{G} \ltimes V$ is linearizable if and only if its monodromy $M(v)$ lies in a single conjugacy class of $H$ for all $v \in V$. 
\end{corollary}

If $D$ is a reductive free divisor in a vector space $X$ (see Section \ref{reductivefreedivisors}), then the twisted fundamental groupoid $\Pi(X,D)$ is precisely a groupoid of the form considered in the present section. Furthermore, $D$ is irreducible if and only if there is a single factor of $\mathbb{C}^*$ in the symmetry group of $D$ (i.e. if $k = 1$ above). Therefore, we obtain a linearization theorem for reductive free divisors. 

\begin{corollary} \label{redfreedivlinearization}
Let $(X, D)$ be a reductive linear free divisor, and let $(P, \phi)$ be an $H$-representation. 
\begin{itemize}
\item If the monodromy of $\phi$ is semisimple, then $\phi$ can be linearized. 
\item Suppose that $D$ is irreducible. Then $\phi$ is linearizable if and only if the value of its monodromy at a point of $X \setminus D$ is conjugate to the value at the origin. 
\end{itemize}
\end{corollary}
\begin{proof}
The only remaining thing to show is that if the monodromy $M_{n}(x)$ at a point of $X \setminus D$  is conjugate to $M_{n}(0)$, then $M_{n}$ lives in a single conjugacy class over all $X$. Let $\mathcal{C} \subset H$ be the conjugacy class of $M_{n}(0)$. Then, by assumption, $M_{n}(x) \in \mathcal{C}$ for all $x \in X \setminus D$. Therefore, for $x \in D$, we have that $M_{n}(x) \in \overline{\mathcal{C}}$, the closure of the conjugacy class. So if $M_{n}(x) \not \in \mathcal{C}$, then it must lie in a conjugacy class $\mathcal{C}'$ of lower dimension. But $0$ lies in the closure of the orbit of $x$, and this implies that $M_{n}(0) \in \overline{\mathcal{C}'}$, which is a contradiction. 
\end{proof}

\section{Classification theorems} \label{normalformtheorems}
In this section, we apply the structure theory developed in Section \ref{ReptheoryWHG} to obtain classification and normal form theorems for representations of positive homogeneous groupoids and their Lie algebroids. In particular, we will obtain normal form theorems for the examples of Sections \ref{algebraicgroupex} and \ref{logconnectionschapt}.

\subsubsection*{Set-up}
We start by describing the general set-up for this section.  Let $(\mathcal{G}, i, \pi)$ be a positive homogeneous Lie groupoid over the vector space $V$. Let $G = (\mathbb{C}^*)^k \times S$ be a reductive Lie group, with $S$ semisimple and simply connected. We assume that this group is equipped with the following: 
\begin{itemize}
\item a linear action on $V$ which extends the given action of $(\mathbb{C}^*)^k$,  
\item an injective homomorphism $r: \tilde{G} \ltimes V \to \mathcal{G}$ extending the morphism $i$, where $\tilde{G}$ is the universal cover of $G$. 
\end{itemize}
We illustrate the set-up with the following diagram of groupoids. 
\[
\begin{tikzpicture}[scale=1.5]
\node (A) at (0,1) {$ \mathbb{C}^k \ltimes V$};
\node (B) at (2,1) {$ \mathbb{C}^k$};
\node (C) at (-1, 0) {$\tilde{G}$};
\node (D) at (0,0) {$\tilde{G} \ltimes V$};
\node (E) at (1,0) {$\mathcal{G}$};

\path[->,>=angle 90]
(A) edge node[above]{$p$} (B)
(D) edge node[below]{$r$} (E)
(A) edge (D)
(A) edge node[above]{$i$} (E)
(E) edge node[below]{$\pi$} (B);

\path[->] (-0.85,0.05) edge node[above]{$\iota$} (-0.35, 0.05);
\path[->] (-0.35, -0.05) edge node[below]{$p$} (-0.85,-0.05);
\end{tikzpicture}
\]

Each groupoid morphism in this diagram induces a pullback functor between categories of representations. Using Theorem \ref{JCdecomposition}, we also have a functor $( \ )_{s} : \Rep(\mathcal{G}, H) \to \Rep(\mathcal{G}, H)$ which extracts the semisimple component of a representation. Define the following \emph{semisimple residue} functor 
\[
Res_{s} : \Rep(\mathcal{G}, H) \to \Rep(\tilde{G}, H), \qquad (P, \phi) \mapsto \iota^{*} r^*(P, \phi_{s}). 
\]
Given a representation $(P, \phi)$, we can trivialise the fibre above the origin $P_{0} \cong H$. This allows us to view the semisimple residue of $\phi$ as a Lie group homomorphism $F: \tilde{G} \to H$. This homomorphism is uniquely determined up to conjugation in $H$. Fixing it, we define $\Rep_{F}(\mathcal{G}, H)$ to be the full subcategory of representations whose semisimple residue is isomorphic to $F$.

Taking the derivative of $F$, we obtain a Lie algebra homomorphism 
\[
dF : \mathfrak{g} \to \mathfrak{h},
\]
where $\mathfrak{g}$ is the Lie algebra of $G$. This Lie algebra decomposes as $\mathfrak{g} = \mathbb{C}^{k} \oplus \mathfrak{s}$, with summands respectively given by the centre and the derived subalgebra. Let $\{ e_{i} \}_{i = 1}^{k}$ denote the standard basis of $\mathbb{C}^{k}$, and let $S_{i} = dF(e_{i})$, which is a family of mutually commuting semisimple elements of $\mathfrak{h}$. Let $\chi: \mathfrak{s} \to \mathfrak{h}$ denote the restriction of $dF$ to the Lie algebra of $S$. Then the $S_{i}$ commute with the elements in the image of $\chi$. 

Since the homogeneous groupoid is positive, there is a factor $\mathbb{C}^* \subseteq (\mathbb{C}^*)^{k}$ which acts on $V$ with strictly positive weights. Fix this factor, let $e \in span_{\mathbb{Z}}(e_{i})$ denote it's infinitesimal generator, and let $D = dF(e) \in \mathfrak{h}$. 

Let $A$ be the Lie algebroid of $\mathcal{G}$, and let $\rho: A \to TV$ denote its anchor map. The map $dr$ allows us to view $\mathfrak{g} \ltimes V$ as a subalgebroid of $A$. Viewing the basis elements $\{ e_{i} \}_{i = 1}^{k} \subset \mathbb{C}^k$ as constant sections of $A$ we have the linear vector fields $E_{i} = \rho(e_{i})$, which are the infinitesimal generators of the $(\mathbb{C}^{*})^{k}$-action on $V$. Let $E = \rho(e)$ be the infinitesimal generator of the chosen $\mathbb{C}^*$-factor. This vector field defines a grading of the functions on $V$, which we call the \emph{$E$-degree}. Note that because all degrees are strictly positive, the subspace of functions with a given $E$-degree is finite dimensional. 

Let $f_{1}, ..., f_{l} \in \mathfrak{s}$ be a basis, and let $[f_{i}, f_{j}] = \sum_{k} c_{ij}^{k} f_{k}$, where $c_{ij}^{k} \in \mathbb{C}$ are the structure constants of the Lie algebra. Viewing the basis elements as sections of $A$, they determine linear vector fields $Y_{i} = \rho(f_{i})$ that generate the action of $S$ on $V$.

Recall from Section \ref{Firstsection} that an $H$-representation $(P,\phi)$ of the groupoid may be differentiated to give a representation of $A$ (also called a flat $A$-connection):
\[
\nabla : A \to At(P).
\]
We assume that the principal $H$-bundle $P$ is trivial. In this case, the connection has the form 
\[
\nabla = d + \omega,
\]
where $\omega \in \Omega_{A}^{1} \otimes \mathfrak{h}$ is the connection $1$-form, which satisfies the following Maurer-Cartan equation 
\begin{equation} \label{MaurerCartanEquation}
d\omega + \frac{1}{2}[\omega, \omega] = 0. 
\end{equation}
Recall that the expression on the left hand side is the curvature of the connection. For a general $1$-form $\omega$ we denote it by $R_{\omega} \in \Omega_{A}^{2} \otimes \mathfrak{h}$. 

In this context, a gauge transformation is a holomorphic map $h: X \to H$, and it acts on the connection $1$-form $\omega$ by sending it to 
\[
Ad_{h}(\omega) + (dh) h^{-1}, 
\]
where $(dh) h^{-1}$ denotes the pullback by $h$ to $\Omega^{1}_{A} \otimes \mathfrak{h}$ of the right-invariant Maurer-Cartan form on $H$. Two connection $1$-forms which are related by a gauge transformation correspond to isomorphic representations. 

\subsection{Structure of positive homogeneous groupoids} \label{StructureofLieAlg}
Positive homogeneous groupoids are rigid, and this allows several of their structures to be put into normal form. Because this simplifies the description of their representations, we derive some of these normal forms in the present section. 

First, consider the homomorphism $\pi \circ r : \tilde{G} \ltimes V \to \mathbb{C}^k$. Restricting it to $\mathbb{C}^{k} \ltimes V$, it gives the projection $p$. Restricting it via $\iota$ to $\tilde{G} = \mathbb{C}^k \times S$, this map must vanish on $S$ since this group is semisimple. On the other hand, the map $\pi \circ r$ might not vanish on $S \ltimes V$, but it is equivalent to a map that does. To see what this means, recall that given a function $f : V \to \mathbb{C}^{k}$, the pullback by the source and target map,
\[
\delta(f) = t^{*}f - s^{*}f: \mathcal{G} \to \mathbb{C}^k,
\]
is a Lie groupoid homomorphism, and this may be added to $\pi$. As long as $f$ is invariant under the $(\mathbb{C}^{*})^{k}$-action, the deformed map $\pi + \delta(f)$ defines a new structure of a homogeneous groupoid. As the result of the following proposition, we will assume that $\pi$ has the property that its restriction to $\tilde{G} \ltimes V$ is the projection onto $\mathbb{C}^{k}$. 

\begin{proposition}
There exists a holomorphic map $f: V \to \mathbb{C}^{k}$ such that $(\mathcal{G}, i, \pi' = \pi + \delta(f))$ is a homogeneous groupoid, and the restriction $\pi' \circ r$ is the projection from $(\mathbb{C}^k \times S) \ltimes V$ to $\mathbb{C}^k$. 
\end{proposition}
\begin{proof}
Consider the representation $\phi_{0} = \exp \circ \pi \circ r : \tilde{G} \ltimes V \to (\mathbb{C}^{*})^{k}$. It has trivial monodromy, and hence by Theorem \ref{mainlinearizationthm}, there is a linearization $\tilde{f}: \phi_{0} \to L(\phi_{0})$. Concretely, the linearization is a map $\tilde{f} : V \to (\mathbb{C}^{*})^{k}$, and because $\pi \circ r \circ \iota$ is the projection of $\mathbb{C}^{k} \times S$ onto $\mathbb{C}^{k}$, it satisfies the following equation 
\[
\phi_{0}(z,s,v) \tilde{f}( (z,s) \ast v) \tilde{f}(v)^{-1} = \exp(z),
\]
for $(z,s,v) \in (\mathbb{C}^k \times S) \ltimes V$. Let $f: V \to \mathbb{C}^{k}$ be a lift of $\tilde{f}$. Then this map has the desired properties. 
\end{proof}

Next, we consider the Lie algebroid $A$. The universal cover $\tilde{G}$ acts on $\mathcal{G}$ by conjugation. By Proposition \ref{liftedactionongpd}, this descends to an action of $G$. Furthermore, this action is by groupoid automorphisms. Differentiating, we obtain an action of $G$ on $A$ by Lie algebroid automorphisms, and this action lifts the linear action of $G$ on $V$. Hence, we may view $A$ as a representation of $G \ltimes V$. Let $\phi: G \ltimes V \to \mathcal{G}(A)$ be the representation, and let $\nabla: \mathfrak{g} \ltimes V \to At(A)$ be the corresponding flat connection. Recall that for a vector bundle $A$, the sections of the Atiyah algebroid $At(A)$ may be identified with the first order differential operators on $A$ whose symbols are vector fields. Given an element $u \in \mathfrak{g}$, viewed as a constant section of $\mathfrak{g} \ltimes V$, and a section $s \in \Gamma(A)$, we have 
\begin{equation} \label{bracketconnectionequation}
\nabla_{u}(s) = [u, s]. 
\end{equation}
The monodromy of the representation $\phi$ is trivial, and so by Theorem \ref{mainlinearizationthm}, $\phi$ may be linearized. However, note that the sub-bundle $\mathfrak{g} \ltimes V$ is preserved by $\phi$, and that this sub-representation is already trivial. Indeed, conjugation by $\tilde{G}$ preserves $\tilde{G} \ltimes V$, and acts on $\mathfrak{g}$ by the adjoint representation. Hence, we would like to find a linearization which is adapted to this sub-representation. 

\begin{definition}
A linearization $T : ( A_{0} \times V, L(\phi) ) \to (A, \phi)$ is adapted to $\mathfrak{g}$ if it restricts to the identity on $\mathfrak{g} \times V$. 
\end{definition}

\begin{lemma} \label{adaptedlinearization}
The representation $(\mathfrak{g} \times V \subseteq A, \phi)$ admits a linearization which is adapted to $\mathfrak{g}$. 
\end{lemma}
\begin{proof}
This proof has many similarities with the proof of Theorem \ref{mainlinearizationthm} and we will mainly note the differences. Let $Fr(A)$ be the frame bundle of $A$, consisting of linear isomorphisms from $A_{0}$ to the fibres of $A$. It is a principal $GL(A_{0})$ bundle, but we treat it as a locally trivial fibre bundle. There is an action of $G \ltimes V$ on $Fr(A)$ given by the following formula 
\[
(g, v) \ast T = \phi(g,v) \circ T \circ L(\phi)(g,v)^{-1},
\]
where $(g,v) \in G \ltimes V$ and $T: A_{0} \to A_{v}$ is a linear isomorphism. Let $GL(\mathfrak{g}, A_{0})$ denote the subgroup of $GL(A_{0})$ consisting of linear transformations that restrict to the identity on $\mathfrak{g}$. There is a reduction of structure $Fr(\mathfrak{g}, A) \subset Fr(A)$ to this subgroup, consisting of linear isomorphisms from $A_{0}$ to the fibres of $A$ that restrict to the identity on $\mathfrak{g}$. This subbundle is preserved by the action of $G \ltimes V$ since both $\phi$ and $L(\phi)$ restrict to the adjoint action on $\mathfrak{g} \times V$. The fibre above $0$ is preserved by the action, and the element $id_{A_{0}} \in Fr(\mathfrak{g}, A)|_{0}$ is a fixed point. Therefore, following the proof of Theorem \ref{mainlinearizationthm}, we may linearize the action in order to produce an equivariant section of the bundle. This is the desired linearization.
\end{proof}

Let $T$ be an adapted linearization produced from Lemma \ref{adaptedlinearization}. We may use this to equip $A_{0} \times V$ with the structure of a Lie algebroid (which is of course isomorphic to $A$). Let $u \in \mathfrak{g}$ and let $a \in A_{0}$, both viewed as constant sections of $A_{0} \times V$. Then computing their bracket, we get 
\[
T([u,a]) = [T(u), T(a)] = [u, T(a)] = \nabla_{u}(T(a)) = T(L(\nabla)_{u}(a)). 
\]
In the first equality, we use the fact that $T$ is a morphism of Lie algebroids, in the second we use the fact that $T|_{\mathfrak{g}} = id$, in the third equality we apply Equation \ref{bracketconnectionequation}, and in the fourth equality, we use that $T$ is a linearization. Therefore, we conclude that the bracket of sections $[u,a]$ is given by a linear action of $\mathfrak{g}$ on $A_{0}$. 

Now choose a decomposition $A_{0} = \mathfrak{g} \oplus W$ as a $G$-representation, and let $\{w_{i}\}_{i = 1}^{d}$ be a basis of $W$ consisting of weight vectors for the $(\mathbb{C}^*)^k$-action. Therefore, we have the following bracket relations 
\[
[e_{i}, w_{j}] = n_{ij} w_{j}, \qquad [e, w_{j}] = m_{j}w_{j},
\]
where $n_{ij}, m_{j} \in \mathbb{Z}$ are integers. 

\begin{corollary}
The Lie algebroid $A$ of $\mathcal{G}$ has a canonical algebraic structure. 
\end{corollary}
\begin{proof}
The proof proceeds by showing that the data defining the Lie algebroid structure on $A_{0} \times V$, namely the bracket and the anchor map, are algebraic. This suffices because this algebroid structure is well-defined up to an automorphism of $L(\phi)$, and we will show in Section \ref{studyingsymmetries} that such automorphisms are algebraic.

For this proof, extend the basis $w_{i}$ of $W$ to a basis of $A_{0}$, keeping the same notation (the elements of $\mathfrak{g}$ have degree $0$ since $e$ lies in the centre of $\mathfrak{g}$). First, applying the anchor map to the equation $[e, w_{i}] = m_{i}w_{i}$ gives $[E, \rho(w_{i})] = m_{i} \rho(w_{i})$, and this implies that $\rho(w_{i})$ is a vector field with polynomial coefficients. Second, taking the bracket of two constant sections $w_{i}$ and $w_{j}$, we get $[w_{i}, w_{j}] = \sum_{k} f_{ij}^{k} w_{k}$, where $f_{ij}^{k}$ are a priori holomorphic functions on $V$. Taking the bracket with $e$, we get on the one hand $[e, [w_{i}, w_{j}]] = (m_{i} + m_{j})[w_{i}, w_{j}]$, and on the other $[e, \sum_{k} f_{ij}^{k} w_{k}] = \sum_{k} (E(f_{ij}^{k}) + m_{k}f_{ij}^k) w_{k}$. Hence $f_{ij}^k$ has fixed $E$-degree $m_{i} + m_{j} - m_{k}$, and so must be a polynomial.
\end{proof}

To finish this section, we record the remaining non-zero brackets. They are given by 
\[
[f_{i}, w_{j}] = \sum_{k} \lambda_{ij}^{k} w_{k}, \qquad [w_{i}, w_{j}] = \sum_{k} \alpha_{ij}^{k}e_{k} + \sum_{k} \beta_{ij}^k f_{k} + \sum_{k} \gamma_{ij}^{k} w_{k},
\]
where $\lambda_{ij}^{k} \in \mathbb{C}$ are the structure constants for the $\mathfrak{s}$-representation on $W$, and $\alpha_{ij}^{k}, \beta_{ij}^{k},$ and $\gamma_{ij}^{k}$ are polynomial functions on $V$. Finally, let $Z_{i} = \rho(w_{i})$, which are vector fields with polynomial coefficients. 

\subsection{Classification of representations} \label{SectionClassofrep}
In this section, we fix a semisimple residue $F$, and give a functorial classification of $\Rep_{F}(\mathcal{G}, H)$, the full subcategory of representations whose semisimple residue is isomorphic to $F$.

Let $X_{F}$ be the set consisting of pairs $(\phi_{s}, U)$, where $\phi_{s} : \mathcal{G} \to Pair(V) \times H$ is a groupoid homomorphism such that $r^{*}\phi_{s} = p^{*}F$, and $U : \mathbb{Z}^k \to Aut(\phi_{s})$ is a homomorphism valued in unipotent automorphisms. Note that the condition on $\phi_{s}$ automatically implies that it has semisimple monodromy. Let $Aut(p^{*}F)$ denote the group of automorphisms of the representation $p^{*}F : \tilde{G} \ltimes V \to H$. There is an action of $Aut(p^*F)$ on $X_{F}$ given by conjugation. Namely, given $h \in Aut(p^*F)$ and $(\phi_{s}, U) \in X_{F}$, we define $h \ast (\phi_{s}, U) = (C_{h} \circ \phi_{s}, C_{h}(U))$, where 
\[
C_{h} \circ \phi_{s}(\gamma) = h(y) \phi_{s}(\gamma) h(x)^{-1}, \qquad C_{h}(U)_{n}(z) = h(z)U_{n}(z)h(z)^{-1}, 
\]
for $\gamma \in \mathcal{G}$ with $t(\gamma) = y$ and $s(\gamma) = x$, and $z \in V$. As a result, we may form an action groupoid \[ Aut(p^*F) \ltimes X_{F}.\] 

\begin{theorem} \label{MainClassificationTheorem}
Let $F : \tilde{G} \to H$ be a Lie group homomorphism with semisimple monodromy, let $\Rep_{F}(\mathcal{G}, H)$ be the full subcategory of representations whose semisimple residue is isomorphic to $F$, and let $Aut(p^*F) \ltimes X_{F}$ be the action groupoid described above. There is an equivalence of categories 
\[
Aut(p^*F) \ltimes X_{F} \cong \Rep_{F}(\mathcal{G}, H). 
\]
\end{theorem}
\begin{proof}
The set $X_{F}$ consists of the objects $(P, \phi_{s}, U) \in \mathcal{JC}_{H}$ with the properties that the underlying principal $H$-bundle is trivial and $r^*\phi_{s} = p^*F$. A morphism in the category with source object $(\phi_{s}, U)$ and target defined on the trivial bundle consists of a holomorphic map $h: V \to H$. The target of such a morphism is $(C_{h} \circ \phi_{s}, C_{h}(U))$. The target is contained in the set $X_{F}$ precisely when $r^*(C_{h} \circ \phi_{s}) = p^*F$. But 
\[
r^*(C_{h} \circ \phi_{s}) = C_{h} \circ (r^* \phi_{s}) = C_{h} \circ p^*F. 
\]
Hence, this is the condition that $h \in Aut(p^*F)$. We conclude that $Aut(p^*F) \ltimes X_{F}$ is isomorphic to a full subcategory of $\Rep_{F}(\mathcal{G}, H)$. 

To finish the proof, it remains for us to show that the inclusion of the subcategory is essentially surjective. Let $(P, \phi) \in \Rep_{F}(\mathcal{G}, H)$ and let $(P, \phi_{s}, U)$ be the corresponding object of $\mathcal{JC}_{H}$. Let $\psi = r^* \phi_{s} \in \Rep( \tilde{G} \ltimes V, H)$. Then $\iota^* \psi \cong F$, implying that $L(\psi) = p^* \iota^* \psi \cong p^* F$. The representation $\psi$ has semisimple monodromy, so by Theorem \ref{mainlinearizationthm} it can be linearized. Hence, it follows that $\psi \cong p^*F$. Choose such an isomorphism $f : (P, \psi) \to (V \times H, p^* F)$. Then $(P, \phi_{s}, U)$ is isomorphic, via $f$, to $(V \times H, C_{f} \circ \phi_{s}, C_{f}(U))$, and 
\[
r^{*}(C_{f} \circ \phi_{s}) = C_{f} \circ (r^* \phi_{s}) = C_{f} \circ \psi = p^* F.
\]
Hence $(C_{f} \circ \phi_{s}, C_{f}(U)) \in X_{F}$. 
\end{proof}

\subsubsection{Symmetries of trivial representations}  \label{studyingsymmetries}
In this section, we fix a semisimple residue $F: \tilde{G} \to H$, and study the group of automorphisms of the trivial representation 
\[
\phi = p^*F : \tilde{G} \ltimes V \to H. 
\]
The automorphism group is defined as follows: 
\[
Aut(\phi) = \{ h: V \to H \ | \ h(gv) \phi(g,v) = \phi(g,v)h(v) \text{ for all } (g,v) \in  \tilde{G} \ltimes V \}. 
\]
In the definition $h$ is assumed to be holomorphic, and the product on the group is the point-wise product of functions. The first thing to note is that because of the positivity assumption, the functions $h$ are actually algebraic. To see this, choose an embedding $H \subseteq GL(\mathbb{C}^M)$ and differentiate the identity $h(ze \ast v) = e^{zD} h(v) e^{-zD}$, which is satisfied by elements of $Aut(\phi)$, to obtain 
\[
E(h) = [D, h]_{c}. 
\]
This implies that the homogeneous components of $h$ are eigenvectors of the linear operator $ad_{D}$, with eigenvalues given by the $E$-degree. Since $ad_{D}$ has finitely many eigenvalues, the components of $h$ can have only finitely many $E$-degrees. Therefore, $h$ must be a polynomial, with a fixed maximal degree determined by $D$. In fact, this implies further that the automorphism group has the structure of an algebraic variety, and we record this as a Lemma. 

\begin{lemma}
The group $Aut(\phi)$ has the structure of an affine algebraic group. 
\end{lemma}
The functor $\iota^*$ defines a homomorphism $Aut(\phi) \to Aut(F)$, sending $h \mapsto h(0)$, and the functor $p^*$ provides a splitting. Both of these maps are algebraic. Let $Aut_{0}(\phi)$ denote the kernel of $\iota^{*}$, consisting of automorphisms $h$ that satisfy $h(0) = id$. Then we have the following split short exact sequence of groups 
\begin{equation} \label{Mainexactsequence}
1 \to Aut_{0}(\phi) \to Aut(\phi) \to Aut(F) \to 1. 
\end{equation}

\begin{lemma}
The group $Aut_{0}(\phi)$ is unipotent. 
\end{lemma}
\begin{proof}
We need to show that all the elements of $Aut_{0}(\phi)$ are unipotent, and this can be done using a faithful representation. Again, we make use of an embedding $H \subseteq GL_{M}(\mathbb{C})$, which allows us to view the elements of $Aut(\phi)$ as polynomial matrices of degree bounded by an integer $N$. Let $I$ denote the ideal of polynomials on $V$ that vanish at the origin, and consider the ring $R = \mathbb{C}[V]/I^{N+1}$. Because of the degree bound on elements of $Aut(\phi)$, there is an injective homomorphism $Aut(\phi) \to GL_{M}(R)$. Since the complex vector space underlying $R$ is finite dimensional, the same is true for the complex vector space underlying $R^{M}$. Hence the forgetful embedding $GL_{M}(R) \to GL(R^{M})$, obtained by viewing an $R$-linear map as a $\mathbb{C}$-linear map, defines a finite-dimensional faithful representation of $Aut(\phi)$. 

An element $h \in Aut_{0}(\phi)$ satisfies $h(0) = id$. Hence, as an element of $End_{R}(R^{M})$, it has the form $h = id + n$, where 
\[
n \in I End_{R}(R^{M}).
\]
The element $n$ is nilpotent, since $I^{N+1} = 0$ in $R$. This implies that $h$ is unipotent. 
 \end{proof}
 
 Now let's turn our attention to $Aut(F)$. This is a subgroup of $H$, defined to be the centralizer of $F(g)$ for all $g \in \tilde{G}$. Hence, it is given by the following intersection 
\[
Aut(F) = C_{H}(S_{1}) \cap ... \cap C_{H}(S_{k}) \cap C_{H}(F|_{S}),
\]
where $C_{H}(S_{i})$ is the centralizer of $S_{i} \in \mathfrak{h}$, and $C_{H}(F|_{S})$ is defined to be the subgroup of elements in $H$ that centralize $F(g)$ for all $g \in S$. 

\begin{proposition} \label{proofofreductive}
The group $Aut(F) \subseteq H$ is a (possibly disconnected) reductive subgroup. Therefore the split short exact sequence \ref{Mainexactsequence} defines a Levi decomposition of $Aut(\phi)$. 
\end{proposition}
\begin{proof}
Consider one of the groups $C_{H}(S_{i})$. Since $S_{i}$ is semisimple, this centralizer agrees with $C_{H}(T(S_{i}))$, where $T(S_{i}) \subset H$ is the torus generated by $S_{i}$. Hence, by Theorem 22.3 and Corollary 26.2A of \cite{humphreys2012linear}, this group is connected and reductive. Since all $S_{i}$ mutually commute, we may therefore inductively prove that $C := C_{H}(S_{1}) \cap ... \cap C_{H}(S_{k}) $ is connected and reductive. Furthermore, the image of $F$ is contained in $C$, and so 
\[
Aut(F) = C_{C}(F|_{S}),
\]
the centralizer of a semisimple subgroup of a reductive group. Let $K = C_{C}(F|_{S})^{o}$ be the connected component of the identity. We will show that it is a reductive group, following \cite{403246}.  In order to do this, we start by choosing a faithful representation $C \subset GL(\mathbb{C}^{M})$. We will make use of the fact that an algebraic subgroup of $GL(\mathbb{C}^{M})$ is reductive if and only if the restriction of the trace form to its Lie algebra is non-degenerate. Let $B$ denote the restriction of the trace form to $Lie(C)$, which is non-degenerate. Viewing $Lie(C)$ as a representation of $S$ (via the homomorphism $F$ and the adjoint representation), the form $B$ is $S$-invariant. Since $S$ is semisimple, $Lie(C)$ decomposes as a direct sum of irreducible subrepresentations: 
\[
Lie(C) = \oplus V_{i}.
\]
Viewed from this perspective, the Lie algebra $Lie(K)$ is the sum of all trivial sub-representations. By Schur's lemma, it follows that $Lie(K)$ is orthogonal to all non-trivial $V_{i}$. Hence, the restriction of $B$ to $Lie(K)$ must be non-degenerate, and so $K$ is reductive. 
\end{proof}

\subsubsection{Algebraic moduli stacks} \label{sectionalgebraicmodulistack}
In this section, we show that $Aut(p^*F) \ltimes X_{F}$ naturally has the structure of an algebraic action groupoid. We will do this by giving a description of $X_{F}$ in terms of flat algebroid connections. 

First, the representation $\phi_{s}$ corresponds, via Lie's theorem, to a Lie algebroid morphism $\nabla : A \to TV \oplus \mathfrak{h}$ whose restriction to $\mathfrak{g} \ltimes V$ is given by $dF$. Namely, the connection is given by $\nabla = d + \omega$, where $\omega$ satisfies $r^{*}\omega = dF$. To be more explicit, let $e^{1}, ..., e^{k}$ denote the dual basis of $(\mathbb{C}^k)^*$, and let $w^{1}, ..., w^{d}$ denote the dual basis of $W^*$. Then $\omega$ has the following form 
\begin{equation} \label{normalformone}
\omega =  \sum_{i = 1}^{k} e^{i} \otimes S_{i} + \chi + \sum_{j = 1}^{d} w^{j} \otimes B_{j},
\end{equation}
where $B_{j}: V \to \mathfrak{h}$ are holomorphic maps. 

The Maurer-Cartan equation for $\nabla$ is the condition that the curvature vanishes: $R_{\omega} = 0$. We start by first imposing a weaker condition. Namely, let $(r \otimes 1): \mathfrak{g} \otimes A \to \wedge^2 A$, and define the following two spaces 
\begin{align*}
W_{F}^{(1)} &= \{ \omega \in \Omega_{A}^1 \otimes \mathfrak{h} \ | \ r^{*}(\omega) = dF, \ \ (r \otimes 1)^{*} R_{\omega} = 0  \},  \\
U_{F} &= \{ B \in \Omega_{A}^1 \otimes \mathfrak{h} \ | \ r^{*}(B) = 0, \ \ (r \otimes 1)^{*} dB + [dF, B] = 0  \}.  
\end{align*}
Let $X_{F}^{(1)}$ denote the space of connections coming from $X_{F}$. It is the subvariety of $W_{F}^{(1)}$ given by forms $\omega$ with vanishing curvature. 
\begin{lemma} \label{Affine1}
The space $U_{F}$ is a finite dimensional vector space, and $W_{F}^{(1)}$ is a non-empty affine space modelled on $U_{F}$. Furthermore, the gauge action of $Aut(p^*F)$ preserves $W_{F}^{(1)}$ and it acts algebraically by affine transformations. 
\end{lemma}
\begin{proof}
Since the equations defining $U_{F}$ are linear, it is clearly a vector space. To see that it is finite dimensional, we express an element in terms of the basis as $B = \sum_{i} w^{i} \otimes B_{i}$, and apply the curvature equation to elements of the form $e \otimes w_{i}$. This gives the following component equations 
\[
E(B_{i}) = m_{i}B_{i} + [B_{i}, D]. 
\]
Each $B_{i}$ is a holomorphic $\mathfrak{h}$-valued function on $V$, and its component of $E$-degree $n$ is an eigenvector for $ad_{D}$ with eigenvalue $n - m_{i}$. Since $ad_{D}$ has finitely many eigenvalues, $B_{i}$ can have non-zero components in only finitely many $E$-degrees. Therefore, it is a polynomial function with a fixed bounded degree determined by $E$ and $D$. It follows from this that $U_{F}$ is finite dimensional. 

It is also clear that $W_{F}^{(1)}$ is an affine space modelled on $U_{F}$. To see that it is non-empty, we observe that setting all $B_{j} = 0$ in Equation \ref{normalformone} gives us a point $\omega_{0} \in W_{F}^{(1)}$. Indeed, it follows from the Lie algebroid bracket relations established in Section \ref{StructureofLieAlg} that the curvature of $\omega_{0}$ is contained in $\wedge^{2}W^{*} \otimes \mathfrak{h}$. 
\end{proof}

The points of $W_{F}^{(1)}$ may be represented in an affine chart by vectors in $U_{F}$ of the form $\sum_{i} w^{i} \otimes B_{i}$. Here the $B_{i}$ are polynomials on $V$ valued in $\mathfrak{h}$ of fixed maximal degree. The components in each degree lie in certain subspaces of $\mathfrak{h}$, and there are linear relations satisfied among the different components and the different $B_{i}$. Hence, by choosing linear coordinates on these subspaces, we obtain affine coordinates on $W_{F}^{(1)}$. The subspace $X_{F}^{(1)}$ is an algebraic subvariety defined by the following system of quadratic equations: 
\begin{equation} \label{quadcurvatureeq}
[B_{i}, B_{j}] + Z_{i}(B_{j}) - Z_{j}(B_{i}) - \sum_{k} \gamma_{ij}^{k}B_{k} = \sum_{k} \alpha_{ij}^{k} S_{k} + \sum_{k} \beta_{ij}^{k} dF(f_{k}),
\end{equation}
where $1 \leq i < j \leq d$. 

Let's now turn to the second component of an element of $X_{F}$, namely the homomorphism $U: \mathbb{Z}^k \to Aut(\phi_{s})$ valued in unipotent automorphisms. This homomorphism may be specified by the $k$ pairwise commuting unipotent automorphisms $U_{e_{i}}$, or if we take logarithms, by $k$ pairwise commuting nilpotent infinitesimal automorphisms $N_{i} \in \mathfrak{aut}(\phi_{s})$. The condition of being an infinitesimal automorphism can be phrased in terms of the connection $\nabla = d + \omega$ as the following equation: 
\[
dN_{i} - [N_{i}, \omega] = 0,
\]
for all $i$. Again, we may analyse the space of such $N_{i}$ by first imposing a weakened condition. Namely, we define the following vector space 
\[
W_{F}^{(2)} = \{ (N_{1}, ..., N_{k}) \in \Omega_{A}^{0} \otimes \mathfrak{h} \ | \ r^* dN_{i} - [N_{i}, dF] = 0 \text{ for all } i \}. 
\]
\begin{lemma} \label{Affine2}
The vector space $W_{F}^{(2)}$ is finite dimensional. Furthermore, the adjoint action of $Aut(p^*F)$ preserves this space and acts algebraically. 
\end{lemma}
\begin{proof}
Applying the defining equation for $W_{F}^{(2)}$ to $e$ and $N_{j}$ gives the following component equation 
\[
E(N_{j}) = [N_{j}, D],
\]
which implies that $N_{j}$ must be a polynomial function with a fixed bounded degree determined by $E$ and $D$. Hence $W_{F}^{(2)}$ is finite dimensional. The rest of the proof then follows easily. 
\end{proof}

Let $W_{F} = W_{F}^{(1)} \times W_{F}^{(2)}$, which is a finite dimensional affine space. The set $X_{F}$ is the algebraic subvariety of $W_{F}$ defined by the Equations \ref{quadcurvatureeq} for $\omega$, as well as the following equations between $\omega$ and $N_{i}$: 
\begin{itemize}
\item $Z_{i}(N_{j}) = [N_{j}, B_{i}]$, for $1 \leq i \leq d$ and $1 \leq j \leq k$,
\item $[N_{i}, N_{j}] = 0$ for $1 \leq i < j \leq k$, 
\item $N_{i}$ is nilpotent for $1 \leq i \leq k$. 
\end{itemize}
Furthermore, Lemmas \ref{Affine1} and \ref{Affine2} imply that there is an algebraic action of $Aut(p^*F)$ on $W_{F}$ which preserves $X_{F}$. Therefore, the groupoid $Aut(p^*F) \ltimes X_{F}$ classifying flat connections is an action groupoid in the category of affine algebraic varieties. We summarize this result in the following theorem. 

\begin{theorem} \label{AlgebraicModuliStack}
Let $A$ be the Lie algebroid of a positive homogenous groupoid, and let $F: \tilde{G} \to H$ be a homomorphism with semisimple monodromy. Then the moduli space of $H$-representations of $A$ with semisimple residue in the adjoint orbit of $dF$ has a natural structure of an affine algebraic quotient stack $[X_{F}/Aut(p^*F)]$. 
\end{theorem}

To finish this section, we apply Theorem \ref{JCdecomposition} to describe the flat connection associated to a point $( \{ B_{i} \}_{i = 1}^{d}, \{ N_{i} \}_{i = 1}^{k} ) \in X_{F}$: 
\[
\nabla_{B,N} = d + \sum_{i=1}^{k} e^{i} \otimes (S_{i} + N_{i}) + \chi + \sum_{j=1}^{d}w^{j} \otimes (B_{j} + \sum_{i=1}^{k}d\pi_{i}(w_{j}) N_{i}),
\]
where $\pi_{i}$ are the components of the map $\pi: \mathcal{G} \to \mathbb{C}^k$.


\subsection{Normal forms for flat algebroid connections} \label{normalformexplicitsection}
In this section, we apply Theorem \ref{MainClassificationTheorem} and the results of Section \ref{SectionClassofrep} to obtain normal form theorems for Lie algebroid representations in the setting of the examples from Sections \ref{algebraicgroupex} and \ref{logconnectionschapt}. In each case, we present a normal form theorem for flat connections with a given fixed semisimple residue $F$, which is itself unique up to the adjoint action in $H$. In other words, we write down a set of equations cutting out the algebraic variety $X_{F}$ of Section \ref{sectionalgebraicmodulistack} which parametrizes flat connections in normal form. The semisimple residue also determines an algebraic group $Aut(p^*F)$, as in Section \ref{studyingsymmetries}, and the normal form is uniquely determined up to the gauge action of this group. Hence, in these examples, we are in effect giving explicit equations for the algebraic moduli stacks of Theorem \ref{AlgebraicModuliStack}. 

\subsubsection{Plane curve singularities} 
In this section, we continue studying the homogeneous plane curve singularities $D \subset \mathbb{C}^2$ considered in Section \ref{planecurvesingularities}. Recall that the logarithmic tangent bundle $T_{\mathbb{C}^{2}}(- \log D)$ is generated by the following vector fields 
\[
E = p x \partial_{x} + q y \partial_{y}, \qquad V = -(\partial_{y}f) \partial_{x} + (\partial_{x} f) \partial_{y}.
\]
The dual \emph{logarithmic} $1$-forms are, respectively,
\[
\alpha = \frac{1}{n} d\log f, \qquad \beta = \frac{1}{nf}(pxdy - qydx).
\]
These forms generate the algebra of logarithmic differential forms $\Omega_{\mathbb{C}^2}^{\bullet}(\log D)$, and they satisfy $d \alpha = 0$ and $d\beta = (n - p - q) \beta \wedge \alpha. $ The connection $1$-form of a logarithmic connection is given by $\omega = A \otimes \alpha + B \otimes \beta,$ where $A, B : \mathbb{C}^2 \to \mathfrak{h}$ are holomorphic. The Maurer-Cartan equation is then given by 
\[
V(A) - E(B) + (n - p - q)B + [B,A] = 0. 
\]
We obtain the following normal form theorem. 

\begin{theorem}\label{planecurvesingnormalform}
Let $\nabla = d + \omega$ be a flat connection on $\mathbb{C}^2$ with logarithmic singularity along the homogeneous plane curve $D$. There is a gauge transformation which puts $\omega$ into the following normal form
\[
\omega = (S + N) \otimes \alpha + B \otimes \beta,
\]
where $S \in \mathfrak{h}$ is a semisimple element, $N : \mathbb{C}^2 \to \mathfrak{h}$ is a holomorphic family of nilpotent elements, and $B: \mathbb{C}^2 \to \mathfrak{h}$ is a holomorphic map. These data satisfy the following equations: 
\begin{align}
E(B) &= (n - p - q)B + [B, S] \label{firsteqpc} \\
E(N) &= [N,S] \label{secondeqpc} \\
V(N) &= [N, B]. \label{thirdeqpc}
\end{align}
The semisimple residue of this connection is given by $S$. 
\end{theorem}

\begin{remark}
In \cite{Narvaez1} the authors classify rank $2$ flat connections on $\mathbb{C}^2$ with regular singularities along the plane curves $x^p - y^q = 0$. The normal forms they obtain are logarithmic flat connections of the form given by Theorem \ref{planecurvesingnormalform}. It would be interesting to compare the classifications of logarithmic and regular singularities in this case. 
\end{remark}

\subsubsection{Algebraic group representations and linear free divisors}
We now turn to the examples studied in Section \ref{algebraicgroupex}, where the positive homogeneous groupoids have the form $\tilde{G} \ltimes V$, where 
\[
\tilde{G} = U \rtimes (\mathbb{C}^k \times S)
\]
is the universal cover of an algebraic group $G$ acting linearly on the vector space $V$. Recall that these groupoids also encompass the linear free divisors of Section \ref{reductivefreedivisors}. The corresponding Lie algebroid is an action algebroid of the form $A = \mathfrak{g} \ltimes V,$ where $\mathfrak{g} = \mathfrak{n} \rtimes (\mathbb{C}^k \oplus \mathfrak{s})$ and where $\mathfrak{n} = Lie(U)$ is a nilpotent Lie algebra, $\mathbb{C}^k = Lie(\mathbb{C}^k)$ is an abelian Lie algebra, and $\mathfrak{s} = Lie(S)$ is a semisimple Lie algebra. The complex of Lie algebroid differential forms is given by 
\[
\Omega_{A}^{\bullet} = \mathcal{O}_{V} \otimes \wedge^{\bullet} \mathfrak{g}^{*}
\]
with differential given by the sum $d = d_{\rho} + d_{CE}$, where $d_{CE}$ is the Chevalley-Eilenberg differential of the Lie algebra $\mathfrak{g}$, and $d_{\rho}$ is defined by the action, which is encoded by the anchor map $\rho : \mathfrak{g} \ltimes V \to TV$. More precisely, $\Omega_{A}^{\bullet}$ is generated by the holomorphic functions on $V$, $f \in \mathcal{O}_{V}$, and the constant elements $\alpha \in \mathfrak{g}^{*}$. The differentials of these forms are given by 
\[
df = d_{\rho}(f) = \rho^{*}(d_{dR}f), \qquad d \alpha = d_{CE}(\alpha),
\]
where $d_{dR}$ is the de Rham differential. Because of the decomposition $\mathfrak{g} = \mathfrak{n} \oplus \mathbb{C}^{k} \oplus \mathfrak{s}$, the differential $d$ decomposes into three components $d = d_{n} + d_{a} + d_{s}$. 

The connection $1$-form of an $A$-connection is given by $\omega = \omega_{n} + \omega_{a} + \omega_{s}$, where
\[
\omega_{n} \in \mathcal{O}_{V} \otimes \mathfrak{n}^{*} \otimes \mathfrak{h}, \ \ \omega_{a} \in \mathcal{O}_{V} \otimes (\mathbb{C}^k)^{*} \otimes \mathfrak{h}, \ \ \omega_{s} \in \mathcal{O}_{V} \otimes \mathfrak{s}^{*} \otimes \mathfrak{h}.
\]
The Maurer-Cartan equation then decomposes into $6$ coupled equations, according to the decomposition of $\wedge^2 \mathfrak{g}^*$ : 
\begin{align*}
d_{\rho, n} \omega_{n} + d_{CE, n}\omega_{n} + \frac{1}{2}[\omega_{n}, \omega_{n}] & = 0 \in \mathcal{O}_{V} \otimes \wedge^2 \mathfrak{n}^* \otimes \mathfrak{h} \\
d_{\rho, a} \omega_{a} +  \frac{1}{2}[\omega_{a}, \omega_{a}] & = 0  \in \mathcal{O}_{V} \otimes \wedge^2 (\mathbb{C}^k)^* \otimes \mathfrak{h} \\
d_{\rho, s} \omega_{s} + d_{CE}\omega_{s} + \frac{1}{2}[\omega_{s}, \omega_{s}] & = 0 \in \mathcal{O}_{V} \otimes \wedge^2 \mathfrak{s}^* \otimes \mathfrak{h} \\
d_{\rho,a} \omega_{n} + d_{\rho, n} \omega_{a} + d_{CE, a} \omega_{n} + [\omega_{n}, \omega_{a}] &= 0 \in \mathcal{O}_{V} \otimes (\mathfrak{n}^* \wedge (\mathbb{C}^k)^*) \otimes \mathfrak{h} \\
d_{\rho, s} \omega_{n} + d_{\rho, n} \omega_{s} + d_{CE, s} \omega_{n} + [\omega_{n}, \omega_{s}] &= 0 \in \mathcal{O}_{V} \otimes (\mathfrak{n}^* \wedge \mathfrak{s}^*) \otimes \mathfrak{h} \\
d_{\rho, s} \omega_{a} + d_{\rho, a}\omega_{s} + [\omega_{s}, \omega_{a}] &= 0 \in \mathcal{O}_{V} \otimes (\mathfrak{s}^* \wedge (\mathbb{C}^k)^*) \otimes \mathfrak{h}
\end{align*}

We obtain a normal form for these flat connections. In order to simplify the description, let $e_{i} \in \mathbb{C}^k$ denote the standard basis, with dual basis $e^{i} \in (\mathbb{C}^k)^{*}$, and let $E_{i} = \rho(e_{i}) \in \mathfrak{X}(V)$ denote the vector fields on $V$. 

\begin{theorem} \label{linearalgrepnormalform}
Let $\nabla = d + \omega$ be a flat $A$-connection. There is a gauge transformation which puts $\omega$ into the following normal form 
\[
\omega = \omega_{n} + \sum_{i = 1}^{k} e^{i} \otimes (S_{i} + N_{i}) + \chi,
\] 
where $\chi \in \mathfrak{s}^* \otimes \mathfrak{h}$ represents a Lie algebra homomorphism, $\omega_{n} \in \mathcal{O}_{V} \otimes \mathfrak{n}^* \otimes \mathfrak{h}$, $S_{i} \in \mathfrak{h}$ are semisimple elements, and $N_{i} \in \mathcal{O}_{V} \otimes \mathfrak{h}$ are holomorphically varying nilpotent elements. These data satisfy the following equations
\begin{enumerate}
\item the $S_{i}$ are mutually commuting, \label{algrepeq1}
\item the $S_{i}$ centralize the image of $\chi$, \label{algrepeq2}
\item $ d_{n} \omega_{n} + \frac{1}{2}[\omega_{n}, \omega_{n}] = 0$, \label{algrepeq3}
\item $ d_{s} \omega_{n} + [\omega_{n}, \chi] = 0$, \label{algrepeq4}
\item $E_{i}(\omega_{n}) = \omega_{n} \circ ad_{e_{i}} - ad_{S_{i}} \circ \omega_{n}$ for all $i = 1, ..., k$, where $E_{i} = \rho(e_{i}) \in TV$, \label{algrepeq5}
\item $E_{j}(N_{i}) = [N_{i}, S_{j}]$ for all $i,j = 1, ..., k$, \label{algrepeq6}
\item $d_{n}N_{i} = ad_{N_{i}} \circ \omega_{n} $ for all $i = 1,..., k$, \label{algrepeq7}
\item $d_{s}N_{i} = ad_{N_{i}} \circ \chi $ for all $i = 1,..., k$, \label{algrepeq8}
\item the $N_{i}$ are mutually commuting, \label{algrepeq9}
\end{enumerate}
The semisimple residue of this connection is given by $\chi$ and the elements $S_{i}$. 
\end{theorem}

For the remainder of this section, we will apply Theorem \ref{linearalgrepnormalform} to the examples of linear free divisors discussed in Section \ref{reductivefreedivisors}. 

\begin{example}[$A_{2}$]
A representation of the action groupoid $\mathbb{C} \ltimes \mathbb{C}$ corresponds to an ordinary differential equation on $\mathbb{C}$  with a Fuchsian singularity at the origin: 
\[
z \frac{ds}{dz} = A(z) s(z). 
\]
In this setting, Theorem \ref{linearalgrepnormalform} states that there is a gauge transformation which puts $A(z)$ into the form $A(z) = S + \sum_{i \geq 0} z^{i}N_{i}$, where $S$ is semisimple, and $N_{i}$ are nilpotent and satisfy $iN_{i} = [N_{i}, S]$. The semisimple residue here is given by $S$, which can be decomposed into its real and imaginary components $S = a + ib$. This normal form is unique up to the gauge action of the group $Aut(p^*S)$. This group was computed in \cite[Section 8.5]{babbitt1983formal} and \cite[Proposition 3.4]{bischoff2020lie}. It is given by the intersection $P(a) \cap C_{H}(\exp(2 \pi i S))$, where $P(a)$ is the parabolic subgroup of $H$ determined by $a$, and $C_{H}(\exp(2 \pi i S))$ is the centralizer of the element $\exp(2 \pi i S) \in H$. This result is well-known in the literature, and various versions appear in \cite{hukuhara1937proprietes, MR5229, levelt1975jordan, turrittin1955convergent, babbitt1983formal, kleptsyn2004analytic, boalch2011riemann, bischoff2020lie}. 
\end{example}

\begin{example}[$A_{k+1}$]
Consider the groupoid described in Example \ref{Aquiver}
\[
\mathbb{C}^k \ltimes \mathbb{C}^k = \Pi(\mathbb{C}^k, D),
\] 
where $D = \{ (z_{1}., ..., z_{k}) \in \mathbb{C}^k \ | \ z_{1} ... z_{k} = 0 \}$ is a normal crossing divisor. Its Lie algebroid is the logarithmic tangent bundle $T_{\mathbb{C}^k}(-\log D)$, which has a global frame given by the vector fields $z_{i} \partial_{z_{i}}$. The dual logarithmic $1$-forms are given by $\frac{dz_{i}}{z_{i}}$, and they generate the complex of logarithmic differential forms. Representations of $T_{\mathbb{C}^k}(-\log D)$ are flat connections on $\mathbb{C}^k$ which have logarithmic singularities along $D$. Applying Theorem \ref{linearalgrepnormalform} we obtain the following normal form for these connections: 
\[
\nabla = d + \sum_{i = 1}^{k} (S_{i} + N_{i}(z))\frac{dz_{i}}{z_{i}},
\]
where $S_{i}$ are pairwise commuting semisimple elements, and 
\[
N_{i} = \sum_{l_{1}, ..., l_{k} \geq 0} z_{1}^{l_{1}}...z_{k}^{l_{k}} N_{i, l_{1}, ..., l_{k}},
\]
with $N_{i, l_{1}, ..., l_{k}}$ nilpotent and satisfying 
\[
[N_{i, l_{1}, ..., l_{k}}, S_{j}] = l_{j} N_{i, l_{1}, ..., l_{k}},
\]
for all $i, l_{1}, ..., l_{k}, j$, and 
\[
\sum_{l_{1} + t_{1} = r_{1}, ..., l_{k} + t_{k} = r_{k}} [N_{i, l_{1}, ..., l_{k}}, N_{j, t_{1}, ..., t_{k}}] = 0,
\]
for all $i,j, r_{1}, ..., r_{k}$. This normal form is unique up to the gauge group of the semisimple residue. Let $S_{i} = a_{i} + ib_{i}$ be the decomposition into real and imaginary parts. Then the gauge group is given by the intersection
\[
P(a_{1}) \cap ... \cap P(a_{n}) \cap C_{H}(\exp(2 \pi i S_{1}), ..., \exp(2 \pi i S_{k})),
\]
where $P(a_{i})$ is the parabolic subgroup of $H$ determined by $a_{i}$, and $C_{H}(\exp(2 \pi i S_{1}), ..., \exp(2 \pi i S_{k}))$ is the centralizer in $H$ of the collection of elements $\exp(2 \pi i S_{i})$. This normal form has appeared in \cite{MR382766, MR436189, MR2085816}.
\end{example}

\begin{example}[$G_{2}$]
Consider the groupoid of Example \ref{G2example}
\[
 (\mathbb{C} \times SL(2,\mathbb{C})) \ltimes S^3(\mathbb{C}^2) = \Pi(\mathbb{C}^4, D),
\]
where $D$ is the vanishing locus of $f(x,y,z,w) = 27w^2x^2 - 18 wxyz + 4wy^3 + 4xz^3 - y^2z^2.$ Its Lie algebroid is the logarithmic tangent bundle $T_{\mathbb{C}^4}(- \log D)$ which has global frame: 
\begin{align*}
 E & = 3(x\partial_{x} + y\partial_{y} + z\partial_{z} + w\partial_{w}), \\
    V_{h} &= 3x\partial_{x} + y\partial_{y} - z\partial_{z} - 3 w\partial_{w}, \\
    V_{f} &= 3x \partial_{y} + 2y \partial_{z} + z \partial_{w}, \\
    V_{e} &= y\partial_{x} + 2z \partial_{y} + 3w\partial_{z}.
\end{align*}
The dual logarithmic $1$-forms are given as follows: 
\begin{align*}
 \alpha_{E} &= \frac{1}{12}d\log(f) \\
    \alpha_{h} &= \frac{1}{2f} \big( (9w^2x - 7wyz + 2z^3)dx + (2wy^2 - yz^2 + 3wxz)dy + (y^2z - 3wxy - 2xz^2)dz \\
    & \qquad \qquad \qquad \qquad \qquad \qquad \qquad \qquad \qquad \qquad \qquad \qquad \qquad + (-2y^3 + 7xyz -9wx^2)dw \big)\\
    \alpha_{f} &= \frac{1}{f} \big( 2(wz^2 - 3w^2y)dx + (9w^2x - wyz)dy + 2(wy^2 - 3wxz)dz + (4xz^2 - 3wxy - y^2z)dw   \big) \\
    \alpha_{e} &= \frac{1}{f} \big((4wy^2 - 3wxz - yz^2)dx + 2(xz^2 - 3wxy)dy + (9wx^2 - xyz)dz + 2(xy^2 - 3x^2z)dw   \big).
\end{align*}
These forms satisfy 
\[
d\alpha_{E} = 0, \ d\alpha_{h} = \alpha_{e} \wedge \alpha_{f}, \ d\alpha_{e} = 2 \alpha_{h} \wedge \alpha_{e}, \ d\alpha_{f} = -2 \alpha_{h} \wedge \alpha_{f}. 
\]
Applying Theorem \ref{linearalgrepnormalform}, we obtain the following normal form for flat connections on $\mathbb{C}^4$ which have logarithmic singularities along $D$: 
\[
\nabla = d + (S + N)\otimes \alpha_{E} + A \otimes \alpha_{h} + B \otimes \alpha_{f} + C \otimes \alpha_{e}, 
\]
where $A, B, C \in \mathfrak{h}$ satisfy 
\[
[B,C] = A, \ \ [A,B] = 2B, \ \ [A,C] = -2C,
\]
$S \in \mathfrak{h}$ is semisimple and commutes with $A,B$ and $C$, and $N : \mathbb{C}^{4} \to \mathfrak{h}$ is a holomorphic map valued in nilpotent elements which satisfies: 
\[
E(N) = [N,S], \ \ V_{h}(N) = [N, A], \ \ V_{f}(N) = [N,B], \ \ V_{e}(N) = [N,C]. 
\]
\end{example}

\begin{example}
Consider the groupoid of Example \ref{symmetricborellinearfree} 
\[
S_{2} \rtimes \tilde{B}_{2} = \Pi(\mathbb{C}^3, D_{2}),
\]
where $D_{2}$ is given by the vanishing of $f_{2} = x(y^2 - xz)$. The Lie algebroid is the logarithmic tangent bundle $T_{\mathbb{C}^3}(-\log D_{2})$, which has global frame consisting of the vector fields 
\[
E_{1} = 2 x \partial_{x} + y \partial_{y}, \ \ E_{2} = y \partial_{y} + 2 z\partial_{z}, \ \ V = x \partial_{y} + 2y \partial_{z}. 
\]
These vector fields satisfy 
\[
[E_{1}, E_{2}] = 0, \ \ [E_{1}, V] = V, \ \ [E_{2}, V] = -V.
\]
The dual logarithmic $1$-forms are 
\[
\alpha_{1} = \frac{1}{2}d\log x, \ \ \alpha_{2} = \frac{1}{2}d\log(xz-y^2) - \frac{1}{2}d\log x , \ \ \beta = \frac{1}{2x(xz-y^2)}(3xzdy - d(xyz)),
\]
which satisfy 
\[
d\alpha_{1} = 0, \ \ d\alpha_{2} = 0, \ \ d \beta = (\alpha_{2} - \alpha_{1}) \wedge \beta. 
\]
Applying Theorem \ref{linearalgrepnormalform}, we obtain the following normal form for flat connections on $\mathbb{C}^3$ which have logarithmic singularities along $D_{2}$: 
\[
\nabla = d + (S_{1} + N_{1}) \otimes \alpha_{1} + (S_{2} + N_{2}) \otimes \alpha_{2} + B \otimes \beta,
\]
where $S_{1}, S_{2} \in \mathfrak{h}$ are commuting semisimple elements, $N_{1}, N_{2} : \mathbb{C}^3 \to \mathfrak{h}$ are holomorphic maps valued in nilpotent elements and which commute, and $B : \mathbb{C}^3 \to \mathfrak{h}$ is a holomorphic map. These data satisfy the following equations 
\begin{align*}
E_{1}(B) &= [B, S_{1}] + B, \ \ E_{2}(B) = [B, S_{2}] - B, \\
E_{i}(N_{j}) &= [N_{j}, S_{i}] \text{ for all } 1 \leq i,j \leq 2, \\
V(N_{1}) &= [N_{1}, B], \ \ V(N_{2}) = [N_{2}, B]. 
\end{align*}
\end{example}

\subsubsection{Sekiguchi's free divisor} \label{Sekinormalform}
In this section, we continue studying the free divisor $D \subseteq \mathbb{C}^3$ considered in Section \ref{SekiDiv}. Recall that it is defined as the vanishing locus of the polynomial $F_{B,5} = xy^4 + y^3z + z^3$, and that the logarithmic tangent bundle $T_{\mathbb{C}^3}(- \log D)$ has a basis given by vector fields $E, V, W$. The dual logarithmic $1$-forms are respectively given by 
\begin{align*}
\alpha &= \frac{y (3y^3 + 4z^2)}{3 F_{B,5}}dx + \frac{z (16 xz - 3y^2)}{3 F_{B,5}}dy + \frac{z^2 + 3y^3 - 12xyz}{3 F_{B,5}}dz \\ 
\beta &= \frac{y^2 z}{6 F_{B,5}}dx + \frac{z(4xy + 3z)}{6 F_{B,5}}dy - \frac{y (3xy + 2z)}{6 F_{B,5}}dz \\
\gamma &= \frac{2y^3 + 3z^2 - 4xyz}{9F_{B,5}} dx - \frac{3yz + xy^2 + 16 x^2 z}{9 F_{B,5}}dy + \frac{12 x^2 y + 2y^2 - xz}{9F_{B,5}}dz. 
\end{align*}
These forms generate the algebra of logarithmic differential forms and they satisfy the following equations 
\[
d \alpha = - 24 z \beta \wedge \gamma, \ \ d\beta = - \alpha \wedge \beta - 6y \beta \wedge \gamma, \ \ d\gamma = -2 \alpha \wedge \gamma + 40 x \beta \wedge \gamma. 
\]
and 
\[
d \log(F_{B,5}) = 9 \alpha -96x \beta - 36y \gamma. 
\]
The connection $1$-form of a logarithmic connection is given by 
\[
\omega = A \otimes \alpha + B \otimes \beta + C \otimes \gamma,
\]
where $A, B, C : \mathbb{C}^3 \to \mathfrak{h}$ are holomorphic. The Maurer-Cartan equation is then given by the following system of coupled equations 
\begin{align*}
& E(B) - V(A) - B + [A,B] = 0 \\
& E(C) - W(A) - 2C + [A,C] = 0 \\ 
& V(C) - W(B) - 24zA - 6yB + 40 x C + [B,C] = 0. 
\end{align*}
We obtain the following normal form theorem. 
\begin{theorem}\label{sekinormalform}
Let $\nabla = d + \omega$ be a flat connection on $\mathbb{C}^3$ with logarithmic singularity along the zero locus of $F_{B,5}$. There is a gauge transformation which puts $\omega$ into the following normal form 
\[
\omega = (S + N) \otimes \alpha + (B - \frac{32}{3}xN) \otimes \beta + (C - 4y N) \otimes \gamma, 
\]
where $S \in \mathfrak{h}$ is a semisimple element, $N: \mathbb{C}^3 \to \mathfrak{h}$ is a holomorphic family of nilpotent elements, and $B, C : \mathbb{C}^3 \to \mathfrak{h}$ are holomorphic maps. These data satisfy the following equations 
\begin{enumerate}
\item $E(B) = B + [B, S]$ \label{seki1}
\item $E(C) = 2C + [C,S] $ \label{seki2} 
\item $E(N) = [N, S]$ \label{seki3}
\item $V(N) = [N,B]$ \label{seki4} 
\item $W(N) = [N,C]$ \label{seki5} 
\item $V(C) - W(B) = 24z S + 6yB - 40 x C - [B,C].$ \label{seki6}
\end{enumerate}
The semisimple residue of this connection is given by $S$. 
\end{theorem}

As a further application of Theorem \ref{sekinormalform} we show that the logarithmic tangent bundle is not an action algebroid. 
\begin{corollary} \label{notaction}
Let $D \subset \mathbb{C}^3$ be the vanishing locus of $F_{B,5} = xy^4 + y^3z + z^3$. Then $T_{\mathbb{C}^3}(-\log D)$ is not isomorphic to an action algebroid. 
\end{corollary}
\begin{proof}
We argue by contradiction. Hence assume that $T_{\mathbb{C}^3}(-\log D) \cong \mathfrak{g} \ltimes \mathbb{C}^3$, for some Lie algebra $\mathfrak{g}$. Since the anchor map vanishes at the origin, $\mathfrak{g}$ is isomorphic to the isotropy Lie algebra $T_{\mathbb{C}^3}(-\log D)|_{0}$. Recall from Section \ref{SekiDiv} that this algebra is the span of $3$ elements $X, Y, Z$ with the following brackets
\[
[X,Y] = Y, \ \ [X, Z] = 2Z, \ \ [Y,Z] = 0. 
\]
Consider the following $2$-dimensional representation of $\mathfrak{g}$: 
\[
X = \begin{pmatrix} 0 & 0 \\ 0 & 1 \end{pmatrix}, \ \ Y = \begin{pmatrix} 0 & 1 \\ 0 & 0 \end{pmatrix}, \ \ Z = \begin{pmatrix} 0 & 0 \\ 0 & 0 \end{pmatrix}. 
\]
Pulling back the representation by the projection morphism $\mathfrak{g} \ltimes \mathbb{C}^3 \to \mathfrak{g}$ gives a $GL(2, \mathbb{C})$-representation of $T_{\mathbb{C}^3}(-\log D)$. This representation has semisimple residue given by $S = X$, and it has a reduction of structure to the subgroup $G_{A}$ integrating the $2$-dimensional non-abelian Lie algebra. We now apply Theorem \ref{sekinormalform} to put this connection into normal form. Hence it is given by matrices $B, C$ and $N$ satisfying a system of equations. Applying Equations \ref{seki1}, \ref{seki2} and \ref{seki3}, along with the value of the connection on the fibre $T_{\mathbb{C}^3}(-\log D)|_{0}$ gives the following 
\[
B = \begin{pmatrix} B_{11}^{(1)} & 1 \\ B_{21}^{(2)} & B_{22}^{(1)} \end{pmatrix}, \ \ C = \begin{pmatrix} C_{11}^{(2)} & C_{12}^{(1)} \\ C_{21}^{(3)} & C_{22}^{(2)} \end{pmatrix}, \ \ N = \begin{pmatrix} 0 & 0 \\ N_{21}^{(1)} & 0 \end{pmatrix},
\]
where the superscripts denote the $E$-degree. By Equation \ref{seki4}, we see that $N = 0$, implying that the monodromy is trivial. The remaining matrices $B$ and $C$ are described by $12$ complex parameters, and the remaining Equation \ref{seki6} translates to $12$ algebraic equations that the parameters must satisfy (Equation \ref{seki5} is trivially satisfied). 

In fact, the linearization for the representation can be chosen so that it respects the reduction of structure to $G_{A}$. Hence, in the expressions for $B$ and $C$, we may assume that $B_{11} = B_{21} = C_{11} = C_{21} = 0$. There now remains only $4$ complex parameters: 
\[
B_{22} = a_{22}x, \ \ C_{12} = dx, \ \ C_{22} = c_{22} x^2 + f_{22}y,
\]
and there remains only $5$ non-trivial equations: 
\[
2d = 6 + f_{22}, \ \ c_{22} = d(40 + a_{22}), \ \ 4c_{22} + 16f_{22} = 6a_{22}, \ \ 2f_{22} - 3a_{22} = 24, \ \ c_{22} = 0. 
\]
This system is now over-determined, and it is straightforward to check that it does not admit any solutions. We thus arrive at a contradiction. 
\end{proof}

 \bibliographystyle{plain}

 \bibliography{bibliography.bib}

\end{document}